\documentclass[12pt,a4paper]{article}
\pdfoutput=1
\usepackage{graphicx}
\usepackage{amsmath}
\usepackage{amsthm}
\usepackage{bm,bbm}
\usepackage[colorlinks]{hyperref}

\newcommand{\vect}[1]{\boldsymbol{#1}}				
\newcommand{\abs}[1]{\lvert#1\rvert}				
\newcommand{\car}[1]{\lvert#1\rvert}				
\newcommand{\floor}[1]{\lfloor#1\rfloor}			
\newcommand{\round}[1]{\lceil#1\rfloor}				
\newcommand{\factrnd}[1]							
	{\mathopen{[\![}#1\mathclose{]\!]}}

\newcommand{\N}{\mathbbm{N}}						
\newcommand{\R}{\mathbbm{R}}						
\newcommand{\Z}{\mathbbm{Z}}						

\newcommand{\U}{\vect{U}}							
\newcommand{\V}{\vect{V}}							
\newcommand{\X}{\vect{X}}							
\newcommand{\Y}{\vect{Y}}							
\newcommand{\Xscaled}{\vect{\mathcal{X}}}			
\newcommand{\Zscaled}{\vect{\mathcal{Z}}}			
\newcommand{\1}{\vect{1}}							

\newcommand{\eps}{\varepsilon}						

\renewcommand{\Pr}{\mathop{\null\mathbbm{P}}}		
\newcommand{\Ex}{\mathop{\null\mathbbm{E}}}			
\newcommand{\I}{\mathop{\null\mathbbm{1}}\nolimits}	
\newcommand{\Law}{\mathop{\null\mathcal{L}}}		

\newcommand{\msc}[1]{\par\small						
 \noindent\textbf{MSC\,2010: }#1}

\newtheorem{theorem}{Theorem}[section]
\newtheorem{lemma}[theorem]{Lemma}
\newtheorem{proposition}[theorem]{Proposition}
\newtheorem{corollary}[theorem]{Corollary}

\theoremstyle{remark}
\newtheorem{remark}[theorem]{Remark}

\makeatletter
	\let\@fnsymbol\@alph
\makeatother

\begin{document}

\title{\vspace{-30pt}Stochastic domination and weak convergence of conditioned 
Bernoulli random vectors}
\author{Erik Broman\footnote{Department of Mathematics, Chalmers University of 
Technology, 412~96 G\"oteborg, Sweden; e-mail: 
\texttt{broman\,(at)\,chalmers.se}}
\and Tim van de Brug\footnote{Department of Mathematics, VU University, De 
Boelelaan 1081a, 1081\,HV Amsterdam, The Netherlands; e-mail: 
\texttt{\char`{t.vande.brug,w.kager,r.w.j.meester\char`}\,(at)\,vu.nl}}
\and Wouter Kager\footnotemark[2] \and Ronald Meester\footnotemark[2]}
\date{\today}

\maketitle

\begin{abstract}
	For $n\geq 1$ let $\X_n$ be a vector of $n$~independent Bernoulli random 
	variables. We assume that $\X_n$ consists of $M$~``blocks'' such that the 
	Bernoulli random variables in block~$i$ have success probability~$p_i$. 
	Here $M$ does not depend on~$n$ and the size of each block is essentially 
	linear in~$n$. Let $\tilde{\X}_n$ be a random vector having the 
	conditional distribution of~$\X_n$, conditioned on the total number of 
	successes being at least~$k_n$, where $k_n$ is also essentially linear 
	in~$n$. Define $\tilde{\Y}_n$ similarly, but with success probabilities 
	$q_i\geq p_i$. We prove that the law of~$\tilde{\X}_n$ converges weakly to 
	a distribution that we can describe precisely. We then prove that $\sup 
	\Pr(\tilde{\X}_n\leq \tilde{\Y}_n)$ converges to a constant, where the 
	supremum is taken over all possible couplings of $\tilde{\X}_n$ 
	and~$\tilde{\Y}_n$. This constant is expressed explicitly in terms of the 
	parameters of the system.
\end{abstract}

\begin{quotation}
	\msc{Primary 60E15, Secondary 60F05}
\end{quotation}

\section[Introduction]{Introduction and main results}
\label{sec:introduction}

Let $\X$ and~$\Y$ be random vectors on~$\R^n$ with respective laws $\mu$ 
and~$\nu$. We say that $\X$ is stochastically dominated by~$\Y$, and write $\X 
\preceq \Y$, if it is possible to define random vectors $\U = (U_1,\dots,U_n)$ 
and $\V = (V_1,\dots,V_n)$ on a common probability space such the laws of $\U$ 
and~$\V$ are equal to $\mu$ and~$\nu$, respectively, and $\U\leq \V$ (that is, 
$U_i\leq V_i$ for all $i\in\{1,\dots,n\}$) with probability~$1$. In this case, 
we also write $\mu \preceq \nu$. For instance, when $\X = (X_1,\dots,X_n)$ and 
$\Y = (Y_1,\dots,Y_n)$ are vectors of $n$~independent Bernoulli random 
variables with success probabilities $p_1,\dots,p_n$ and $q_1,\dots,q_n$, 
respectively, and $0< p_i\leq q_i< 1$ for $i\in\{1,\dots,n\}$, we have $\X 
\preceq \Y$.

In this paper, we consider the \emph{conditional} laws of $\X$ and~$\Y$, 
conditioned on the total number of successes being at least~$k$, or sometimes 
also equal to~$k$, for an integer~$k$. In this first section, we will state 
our main results and provide some intuition. All proofs are deferred to later 
sections.

Domination issues concerning the conditional law of Bernoulli vectors 
conditioned on having at least a certain number of successes have come up in 
the literature a number of times. In \cite{BHS} and~\cite{BM}, a simplest case 
has been considered in which $p_i=p$ and $q_i=q$ for some $p<q$. In~\cite{BM}, 
the conditional domination is used as a tool in the study of random trees.

Here we study such domination issues in great detail and generality. The 
Bernoulli vectors we consider have the property that the $p_i$ and~$q_i$ take 
only finitely many values, uniformly in the length~$n$ of the vectors. The 
question about stochastic ordering of the corresponding conditional 
distributions gives rise to a number of intriguing questions which, as it 
turns out, can actually be answered. Our main result, 
Theorem~\ref{thm:asymptstochdom}, provides a complete answer to the question 
with what maximal probability two such conditioned Bernoulli vectors can be 
ordered in any coupling, when the length of the vectors tends to infinity.

In Section~\ref{ssec:stochDomination}, we will first discuss domination issues 
for finite vectors $\X$ and~$\Y$ as above. In order to deal with domination 
issues as the length~$n$ of the vectors tends to infinity, it will be 
necessary to first discuss weak convergence of the conditional distribution of 
a single vector. Section~\ref{ssec:framework} introduces the framework for 
dealing with vectors whose lengths tend to infinity, and 
Section~\ref{ssec:weakConvergence} discusses their weak convergence. Finally, 
Section~\ref{ssec:asymptStochDomination} deals with the asymptotic domination 
issue when $n\to \infty$.

\subsection[Domination: finite case]{Stochastic domination of finite vectors}
\label{ssec:stochDomination}

As above, let $\X = (X_1,\dots,X_n)$ and $\Y = (Y_1,\dots,Y_n)$ be vectors of 
independent Bernoulli random variables with success probabilities 
$p_1,\dots,p_n$ and $q_1,\dots,q_n$, respectively, where $0< p_i\leq q_i< 1$ 
for $i\in\{1,\dots,n\}$. For an event~$A$, we shall denote by $\Law(\X|A)$ the 
conditional law of~$\X$ given~$A$. Our first proposition states that the 
conditional law of the total number of successes of~$\X$, conditioned on the 
event $\{\sum_{i=1}^n X_i \geq k\}$, is stochastically dominated by the 
conditional law of the total number of successes of~$\Y$.

\begin{proposition}\label{prop:sums}
	For all $k\in\{0,1,\dots,n\}$,
	\[
		\textstyle \Law(\sum_{i=1}^n X_i | \sum_{i=1}^n X_i \geq k)
		\preceq \Law(\sum_{i=1}^n Y_i | \sum_{i=1}^n Y_i \geq k).
	\]
\end{proposition}

In general, the conditional law of the full vector~$\X$ is not necessarily 
stochastically dominated by the conditional law of the vector~$\Y$. For 
example, consider the case $n=2$, $p_1=p_2=q_1=p$ and $q_2=1-p$ for some 
$p<\tfrac12$, and $k=1$. We then have
\begin{align*}
	\Pr(X_1 = 1 \mid X_1+X_2 \geq 1) &= \frac{1}{2-p},\\
	\Pr(Y_1 = 1 \mid Y_1+Y_2 \geq 1) &= \frac{p}{1-(1-p)p}.
\end{align*}
Hence, if $p$ is small enough, then the conditional law of~$\X$ is not 
stochastically dominated by the conditional law of~$\Y$.

We would first like to study under which conditions we do have stochastic 
ordering of the conditional laws of $\X$ and~$\Y$. For this, it turns out to 
be very useful to look at the conditional laws of $\X$ and~$\Y$, conditioned 
on the total number of successes being \emph{exactly equal} to~$k$, for an 
integer~$k$. Note that if we condition on the total number of successes being 
exactly equal to~$k$, then the conditional law of~$\X$ is stochastically 
dominated by the conditional law of~$\Y$ if and only if the two conditional 
laws are equal. The following proposition characterizes stochastic ordering of 
the conditional laws of $\X$ and~$\Y$ in this case. First we define, for $i\in 
\{1,\dots,n\}$,
\begin{equation}\label{eqn:defbeta}
	\beta_i := \frac{p_i}{1-p_i} \frac{1-q_i}{q_i}.
\end{equation}
The $\beta_i$ will play a crucial role in the domination issue throughout the 
paper.

\begin{proposition}\label{prop:equallaws}
	The following statements are equivalent:
	\begin{itemize}
		\item[(i)] All $\beta_i$ ($i\in \{1,\dots,n\}$) are equal;
		\item[(ii)] $\Law(\X | \sum_{i=1}^n X_i = k) = \Law(\Y | \sum_{i=1}^n 
			Y_i = k)$ for all $k\in\{0,1,\dots,n\}$;
		\item[(iii)] $\Law(\X | \sum_{i=1}^n X_i = k) = \Law(\Y | \sum_{i=1}^n 
			Y_i = k)$ for some $k\in\{1,\dots,n-1\}$.
	\end{itemize}
\end{proposition}

We will use this result to prove the next proposition, which gives a 
sufficient condition under which the conditional law of~$\X$ is stochastically 
dominated by the conditional law of~$\Y$, in the case when we condition on the 
total number of successes being at least~$k$.

\begin{proposition}\label{prop:stochdom}
	If all $\beta_i$ ($i\in \{1,\dots,n\}$) are equal, then for all 
	$k\in\{0,1,\dots,n\}$,
	\[
		\textstyle \Law(\X | \sum_{i=1}^n X_i \geq k)
		\preceq \Law(\Y | \sum_{i=1}^n Y_i \geq k).
	\]
\end{proposition}

The condition in this proposition is a sufficient condition, not a necessary 
condition. For example, if $n=2$, $p_1=p_2=\tfrac12$, $q_1=\tfrac{6}{10}$ and 
$q_2=\tfrac{7}{10}$, then $\beta_1\neq \beta_2$, but we do have stochastic 
ordering for all $k\in\{0,1,2\}$.

\subsection[Asymptotic framework]{Framework for asymptotic domination}
\label{ssec:framework}

Suppose that we now extend our Bernoulli random vectors $\X$ and~$\Y$ to 
infinite sequences $X_1,X_2,\dotsc$ and $Y_1,Y_2,\dotsc$ of independent 
Bernoulli random variables, which we assume to have only finitely many 
distinct success probabilities. It then seems natural to let $\X_n$ and~$\Y_n$ 
denote the $n$-dimensional vectors $(X_1,\dots,X_n)$ and $(Y_1,\dots,Y_n)$, 
respectively, and consider the domination issue as $n\to \infty$, where we 
condition on the total number of successes being at least $k_n = \floor{\alpha 
n}$ for some fixed number $\alpha\in (0,1)$.

More precisely, with $k_n$ as above, let $\tilde\X_n$ be a random vector 
having the law $\Law(\X_n | \sum_{i=1}^n X_i \geq k_n)$, and define 
$\tilde\Y_n$ similarly. Proposition~\ref{prop:stochdom} gives a sufficient 
condition under which $\tilde \X_n$ is stochastically dominated by~$\tilde 
\Y_n$ for each $n\geq1$. If this condition is not fulfilled, however, we might 
still be able to define random vectors $\U$ and~$\V$, with the same laws as 
$\tilde \X_n$ and~$\tilde \Y_n$, on a common probability space such that the 
probability that $\U\leq \V$ is high (perhaps even~$1$). We denote by
\begin{equation}\label{eqn:sup}
	\sup \Pr(\tilde \X_n\leq \tilde \Y_n)
\end{equation}
the supremum over all possible couplings $(\U,\V)$ of $(\tilde \X_n, \tilde 
\Y_n)$ of the probability that $\U \leq \V$. We want to study the asymptotic 
behaviour of this quantity as $n\to\infty$.

As an example (and an appetizer for what is to come), consider the following 
situation. For $i\geq 1$ let the random variable~$X_i$ have success 
probability~$p$ for some $p\in (0,\tfrac12)$. For $i\geq 1$ odd or even let 
the random variable~$Y_i$ have success probability $p$ or $1-p$, respectively. 
We will prove that  $\sup \Pr(\tilde \X_n \leq \tilde \Y_n)$ converges to a 
constant as $n\to \infty$ (Theorem~\ref{thm:asymptstochdom} below). It turns 
out that there are three possible values of the limit, depending on the value 
of~$\alpha$:
\begin{itemize}
	\item[(i)] If $\alpha < p$, then $\sup\Pr(\tilde \X_n\leq \tilde \Y_n) \to 
		1$.
	\item[(ii)] If $\alpha = p$, then $\sup\Pr(\tilde \X_n\leq \tilde \Y_n) 
		\to \tfrac{3}{4}$.
	\item[(iii)] If $\alpha > p$, then $\sup\Pr(\tilde \X_n\leq \tilde \Y_n) 
		\to 0$.
\end{itemize}

In fact, to study the asymptotic domination issue, we will work in an even 
more general framework, which we shall describe now. For every $n\geq1$, 
$\X_n$ is a vector of $n$ independent Bernoulli random variables. We assume 
that this vector is organized in $M$~``blocks'', such that all Bernoulli 
variables in block~$i$ have the same success probability~$p_i$, for $i\in 
\{1,\dots,M\}$. Similarly, $\Y_n$ is a vector of $n$ independent Bernoulli 
random variables with the exact same block structure as~$\X_n$, but 
for~$\Y_n$, the success probability corresponding to block~$i$ is~$q_i$, where 
$0<p_i\leq q_i<1$ as before.

For given $n\geq1$ and $i\in \{1,\dots,M\}$, we denote by~$m_{in}$ the size of 
block~$i$, where of course $\sum_{i=1}^M m_{in} = n$. In the example above, 
there were two blocks, each containing (roughly) one half of the Bernoulli 
variables, and the size of each block was increasing with~$n$. In the general 
framework, we only assume that the fractions $m_{in}/n$ converge to some 
number $\alpha_i\in (0,1)$ as $n\to\infty$, where $\sum_{i=1}^M \alpha_i = 1$. 
Similarly, in the example above we conditioned on the total number of 
successes being at least~$k_n$, where $k_n = \floor{\alpha n}$ for some fixed 
$\alpha\in (0,1)$. In the general framework, we only assume that we are given 
a fixed sequence of integers~$k_n$ such that $0\leq k_n\leq n$ for 
all~$n\geq1$ and $k_n/n\to \alpha\in (0,1)$ as $n\to \infty$.

In this general framework, let $\tilde \X_n$ be a random vector having the 
conditional distribution of~$\X_n$, conditioned on the total number of 
successes being at least~$k_n$. Observe that given the number of successes in 
a particular block, these successes are uniformly distributed within the 
block. Hence, the distribution of~$\tilde \X_n$ is completely determined by 
the distribution of the $M$-dimensional vector describing the numbers of 
successes per block. Therefore, before we proceed to study the asymptotic 
behaviour of the quantity~\eqref{eqn:sup}, we shall first study the asymptotic 
behaviour of this $M$-dimensional vector.

\subsection{Weak convergence}
\label{ssec:weakConvergence}

Consider the general framework introduced in the previous section. We define 
$X_{in}$ as the number of successes of the vector~$\X_n$ in block~$i$ and 
write $\Sigma_n := \sum_{i=1}^M X_{in}$ for the total number of successes 
in~$\X_n$. Then $X_{in}$ has a binomial distribution with parameters $m_{in}$ 
and~$p_i$ and, for fixed~$n$, the~$X_{in}$ are independent. In this section, 
we shall study the joint convergence in distribution of the~$X_{in}$ as $n\to 
\infty$, conditioned on $\{\Sigma_n \geq k_n\}$, and also conditioned on 
$\{\Sigma_n = k_n\}$.

First we consider the case where we condition on $\{\Sigma_n = k_n\}$. We will 
prove (Lemma~\ref{lem:centres} below) that the~$X_{in}$ concentrate around the 
values $c_{in} m_{in}$, where the~$c_{in}$ are determined by the system of 
equations
\begin{equation}\label{eqn:centres}
	\left\{ \begin{aligned}
		&\frac{1-c_{in}}{c_{in}} \frac{p_i}{1-p_i}
		= \frac{1-c_{jn}}{c_{jn}} \frac{p_j}{1-p_j}
		&&& \forall i,j\in\{1,\dots,M\}; \\
		& \textstyle\sum_{i=1}^M c_{in} m_{in} = k_n.
	\end{aligned} \right.
\end{equation}
We will show in Section~\ref{sec:weakConvergence} that the 
system~\eqref{eqn:centres} has a unique solution and that
\[
	c_{in}\to c_i \qquad \text{as $n\to\infty$},
\]
for some $c_i$ strictly between $0$ and~$1$. As we shall see, each 
component~$X_{in}$ is roughly normally distributed around the central value 
$c_{in} m_{in}$, with fluctuations around this centre of the order~$\sqrt{n}$. 
Hence, the proper scaling is obtained by looking at the $M$-dimensional vector
\begin{equation}\label{eqn:defXnintro}
	\Xscaled_n := \left( \frac{X_{1n}-c_{1n} m_{1n}}{\sqrt n}, 
	\frac{X_{2n}-c_{2n} m_{2n}}{\sqrt n}, \dots, \frac{X_{Mn}-c_{Mn} 
	m_{Mn}}{\sqrt n} \right).
\end{equation}

Since we condition on $\{\Sigma_n = k_n\}$, this vector is essentially an 
$(M-1)$-dimensional vector, taking only values in the hyperplane
\[
	S_0 := \{(z_1,\dots,z_M)\in \R^M\colon z_1+\dots+z_M = 0\}.
\]
However, we want to view it as an $M$-dimensional vector, mainly because when 
we later condition on $\{\Sigma_n \geq k_n\}$, $\Xscaled_n$ will no longer be 
restricted to a hyperplane. One expects that the laws of the~$\Xscaled_n$ 
converge weakly to a distribution which concentrates on~$S_0$ and is, 
therefore, singular with respect to $M$-dimensional Lebesgue measure. To 
facilitate this, it is natural to define a measure~$\nu_0$ on the Borel sets 
of~$\R^M$ through
\begin{equation}\label{eqn:defnu0}
	\nu_0(\,\cdot\,) := \lambda_0(\,\cdot\, \cap S_0),
\end{equation}
where $\lambda_0$ denotes ($(M-1)$-dimensional) Lebesgue measure on~$S_0$, and 
to identify the weak limit of the~$\Xscaled_n$ via a density with respect 
to~$\nu_0$. The density of the weak limit is given by the function $f\colon 
\R^M\to\R$ defined by
\begin{equation}\label{eqn:deff}
	f(z) = \I_{S_0}(z) \prod_{i=1}^M \exp\left( -\frac{z_i^2}
	{2c_i(1-c_i)\alpha_i} \right).
\end{equation}

\begin{theorem}\label{thm:limitgivensuccesses}
	The laws $\Law(\Xscaled_n | \Sigma_n = k_n)$ converge weakly to the 
	measure which has density $f / \int f\,d\nu_0$ with respect to~$\nu_0$.
\end{theorem}

We now turn to the case where we condition on $\{\Sigma_n\geq k_n\}$. Our 
strategy will be to first study the case where we condition on the event 
$\{\Sigma_n = k_n+\ell\}$, for $\ell\geq0$, and then sum over~$\ell$. We will 
calculate the relevant range of~$\ell$ to sum over. In particular, we will 
show that for large enough~$\ell$ the probability $\Pr(\Sigma_n = k_n+\ell)$ 
is so small, that these~$\ell$ do not have a significant effect on the 
conditional distribution of~$\Xscaled_n$. For $k_n$ sufficiently larger 
than~$\Ex(\Sigma_n)$, only~$\ell$ of order $o(\sqrt n)$ are relevant, which 
leads to the following result:

\begin{theorem}\label{thm:limitabovemean}
	If $\alpha > \sum_{i=1}^M p_i\alpha_i$ or, more generally, $(k_n - 
	\Ex(\Sigma_n)) / \sqrt{n} \to \infty$, then the laws $\Law(\Xscaled_n | 
	\Sigma_n \geq k_n)$ also converge weakly to the measure which has density 
	$f / \int f\,d\nu_0$ with respect to~$\nu_0$.
\end{theorem}

Finally, we consider the case where we condition on $\{\Sigma_n\geq k_n\}$ 
with $k_n$ below or around $\Ex(\Sigma_n)$, that is, when $(k_n - 
\Ex(\Sigma_n)) / \sqrt{n}\to K\in [-\infty, \infty)$. An essential difference 
compared to the situation in Theorem~\ref{thm:limitabovemean}, is that the 
probabilities of the events $\{\Sigma_n \geq k_n\}$ do not converge to~$0$ in 
this case, but to a strictly positive constant. In this situation, the right 
vector to look at is the $M$-dimensional vector
\[
	\Xscaled^p_n := \left( \frac{X_{1n}-p_1 m_{1n}}{\sqrt n}, \frac{X_{2n}-p_2 
	m_{2n}}{\sqrt n}, \dots, \frac{X_{Mn}-p_M m_{Mn}}{\sqrt n} \right).
\]
It follows from standard arguments that the unconditional laws 
of~$\Xscaled^p_n$ converge weakly to a multivariate normal distribution with 
density $h / \int h d\lambda$ with respect to $M$-dimensional Lebesgue 
measure~$\lambda$, where $h\colon \R^M\to\R$ is given by
\begin{equation}\label{eqn:defh}
	h(z) =  \prod_{i=1}^M \exp\left( -\frac{z_i^2} {2p_i(1-p_i)\alpha_i} 
	\right).
\end{equation}
If $k_n$ stays sufficiently smaller than $\Ex(\Sigma_n)$, that is, when 
$K=-\infty$, then the effect of conditioning vanishes in the limit, and the 
conditional laws of~$\Xscaled^p_n$ given $\{\Sigma_n\geq k_n\}$ converge 
weakly to the same limit as the unconditional laws of~$\Xscaled^p_n$. In 
general, if $K\in [-\infty, \infty)$, the conditional laws of~$\Xscaled^p_n$ 
given $\{\Sigma_n\geq k_n\}$ converge weakly to the measure which has, up to a 
normalizing constant, density~$h$ restricted to the half-space
\begin{equation}\label{eqn:defHK}
	H_K := \{(z_1,\dots,z_M) \in \R^M\colon z_1+\dots+z_M \geq K\}.
\end{equation}

\begin{theorem}\label{thm:limitaroundmean}
	If $(k_n-\Ex(\Sigma_n))/\sqrt n\to K$ for some $K\in[-\infty,\infty)$, 
	then the laws $\Law(\Xscaled^p_n | \Sigma_n \geq k_n)$ converge weakly to 
	the measure which has density $h \I_{H_K} / \int h \I_{H_K}\, d\lambda$ 
	with respect to~$\lambda$.
\end{theorem}

\begin{remark}
	If $(k_n-\Ex(\Sigma_n))/\sqrt n$ does not converge as $n\to\infty$ and 
	does not diverge to either $\infty$ or~$-\infty$, then the laws 
	$\Law(\Xscaled^p_n | \Sigma_n \geq k_n)$ do not converge weakly either. 
	This follows from our results above by considering limits along different 
	subsequences of the~$k_n$.
\end{remark}

\subsection[Asymptotic domination]{Asymptotic stochastic domination}
\label{ssec:asymptStochDomination}

Consider again the general framework for vectors $\X_n$ and~$\Y_n$ introduced 
in Section~\ref{ssec:framework}. Recall that we write $\tilde\X_n$ for a 
random vector having the conditional distribution of the vector~$\X_n$, given 
that the total number of successes is at least~$k_n$. For $n\geq 1$ and $i\in 
\{1,\dots,M\}$, we let $\tilde X_{in}$ denote the number of successes 
of~$\tilde \X_n$ in block~$i$. We define $\tilde \Y_n$ and~$\tilde Y_{in}$ 
analogously. We want to study the asymptotic behaviour as $n\to\infty$ of the 
quantity
\[
	\sup \Pr(\tilde \X_n \leq \tilde \Y_n),
\]
where the supremum is taken over all possible couplings of $\tilde \X_n$ 
and~$\tilde \Y_n$.

Define $\beta_i$ for $i\in \{1,\dots,M\}$ as in~\eqref{eqn:defbeta}. As a 
first observation, note that if all~$\beta_i$ are equal, then by 
Proposition~\ref{prop:stochdom} we have $\sup \Pr(\tilde\X_n \leq \tilde\Y_n) 
= 1$ for every $n\geq 1$. Otherwise, under certain conditions on the 
sequence~$k_n$, $\sup \Pr(\tilde\X_n \leq \tilde\Y_n)$ will converge to a 
constant as $n\to\infty$, as we shall prove.

The intuitive picture behind this is as follows. Without conditioning, $\X_n 
\preceq \Y_n$ for every $n\geq 1$. Now, as long as~$k_n$ stays significantly 
smaller than~$\Ex(\Sigma_n)$, the effect of conditioning will vanish in the 
limit, and hence we can expect that $\sup \Pr(\tilde\X_n \leq \tilde\Y_n)\to 
1$ as $n\to\infty$. Suppose now that we start making the~$k_n$ larger. This 
will increase the number of successes~$\tilde X_{in}$ of the 
vector~$\tilde\X_n$ in each block~$i$, but as long as~$k_n$ stays below the 
expected total number of successes of~$\Y_n$, increasing~$k_n$ will not change 
the numbers of successes per block significantly for the vector~$\tilde\Y_n$.

At some point, when~$k_n$ becomes large enough, there will be a block~$i$ such 
that $\tilde X_{in}$ becomes roughly equal to~$\tilde Y_{in}$. We shall  see 
that this happens for~$k_n$ ``around'' the value~$\hat k_n$ defined by
\[
	\hat k_n := \sum_{i=1}^M \frac{p_i m_{in}}{p_i + \beta_{\max} (1-p_i)},
\]
where $\beta_{\max} := \max\{\beta_1,\dots,\beta_M\}$. Therefore, the 
sequence~$\hat k_n$ will play a key role in our main result. What will happen 
is that as long as~$k_n$ stays significantly smaller than~$\hat k_n$, $\tilde 
X_{in}$ stays significantly smaller than~$\tilde Y_{in}$ for each block~$i$, 
and hence $\sup \Pr(\tilde\X_n \leq \tilde\Y_n)\to 1$ as $n\to\infty$. For 
$k_n$ around~$\hat k_n$ there is a ``critical window'' in which interesting 
things occur. Namely, when $(k_n - \hat k_n) / \sqrt{n}$ converges to a finite 
constant~$K$, $\sup\Pr(\tilde \X_n \leq \tilde \Y_n)$ converges to a 
constant~$P_K$ which is strictly between $0$ and~$1$. Finally, when $k_n$ is 
sufficiently larger than~$\hat k_n$, there will always be a block~$i$ such 
that $\tilde X_{in}$ is significantly larger than~$\tilde Y_{in}$. Hence, 
$\sup\Pr(\tilde\X_n \leq \tilde\Y_n)\to 0$ in this case.

Before we state our main theorem which makes this picture precise, let us 
first define the non-trivial constant~$P_K$ which occurs as the limit of 
$\sup\Pr(\tilde \X_n \leq \tilde \Y_n)$ when~$k_n$ is in the critical window. 
To this end, let
\[
	I := \{i\in \{1,\dots,M\}\colon \beta_i = \beta_{\max}\},
\]
and define positive numbers $a$, $b$ and~$c$ by
\begin{subequations}\label{eqn:defabc}
\begin{align}
	a^2 &= \sum_{i\in I} \frac{\beta_{\max} p_i (1-p_i) \alpha_i}{(p_i + 
	\beta_{\max} (1-p_i))^2} = \sum_{i\in I} q_i (1-q_i) \alpha_i; 
	\label{eqn:defa} \\
	b^2 &= \sum_{i\notin I} \frac{\beta_{\max} p_i (1-p_i) \alpha_i}{(p_i + 
	\beta_{\max} (1-p_i))^2}; \label{eqn:defb} \\
	c^2 &= a^2 + b^2. \label{eqn:defc}
\end{align}
\end{subequations}
As we shall see later, these numbers will come up as variances of certain 
normal distributions. Let $\Phi\colon \R\to (0,1)$ denote the distribution 
function of the standard normal distribution. For $K\in\R$, define $P_K$ by
\begin{equation}\label{eqn:defPK}
	P_K = \begin{cases}
		\rule[-5ex]{0pt}{0pt}\displaystyle
		1 - \int_{-\infty}^{\frac{c-b}{ac} K} \frac{e^{-z^2/2}}{\sqrt{2\pi}} 
		\frac{\Phi\bigl( \frac{K-az}{b} \bigr) - \Phi\bigl( \frac{K}{c} 
		\bigr)}{1-\Phi\bigl( \frac{K}{c} \bigr)} \, dz
		& \text{if $\alpha = \sum_{i=1}^M p_i \alpha_i$}, \\
		\displaystyle
		\Phi\left( \frac{bK}{ac} - \frac{1}{a} R_K \right) + \Phi\left( - 
		\frac{K}{a} + \frac{b}{ac} R_K \right)
		& \text{if $\alpha > \sum_{i=1}^M p_i \alpha_i$}.
	\end{cases}
\end{equation}
where $R_K = \sqrt{K^2 + c^2 \log (c^2/b^2)}$. It will be made clear in 
Section~\ref{sect:asymptStochDomination} where these formulas for~$P_K$ come 
from. We will show that $P_K$ is strictly between $0$ and~$1$. In fact, it is 
possible to show that both expressions for~$P_K$ are strictly decreasing 
in~$K$ from $1$ to~$0$, but we omit the (somewhat lengthy) derivation of this 
fact here.

\begin{theorem}\label{thm:asymptstochdom}
	If all $\beta_i$ ($i\in \{1,\dots,M\}$) are equal, then we have that $\sup 
	\Pr(\tilde \X_n \leq \tilde \Y_n) = 1$ for every $n\geq 1$. Otherwise, the 
	following holds:
	\begin{itemize}
		\item[(i)] If $(k_n - \hat k_n)/\sqrt{n}\to -\infty$, then $\sup 
			\Pr(\tilde \X_n \leq \tilde \Y_n)\to 1$.
		\item[(ii)] If $(k_n - \hat k_n)/\sqrt{n}\to K$ for some $K\in\R$, 
			then $\sup \Pr(\tilde \X_n \leq \tilde \Y_n)\to P_K$.
		\item[(iii)] If $(k_n - \hat k_n)/\sqrt{n}\to \infty$, then $\sup 
			\Pr(\tilde \X_n \leq \tilde \Y_n)\to 0$.
	\end{itemize}
\end{theorem}

\begin{remark}
	If $\beta_i\neq \beta_j$ for some $i\neq j$, and $(k_n - \hat k_n) / 
	\sqrt{n}$ does not converge as $n\to\infty$ and does not diverge to either 
	$\infty$ or~$-\infty$, then $\sup \Pr(\tilde \X_n \leq \tilde \Y_n)$ does 
	not converge either. This follows from the strict monotonicity of~$P_K$, 
	by considering the limits along different subsequences of the~$k_n$.
\end{remark}

To demonstrate Theorem~\ref{thm:asymptstochdom}, recall the example from 
Section~\ref{ssec:framework}. Here $\beta_{\max} = 1$, $\hat k_n = p n$, $I = 
\{1\}$ and $a^2 = b^2 = \tfrac{1}{2} p (1-p)$. If $\alpha = p$, then we have 
that $(k_n - \hat k_n) / \sqrt{n} \to 0$ as $n\to\infty$. Hence, by 
Theorem~\ref{thm:asymptstochdom}, $\sup\Pr(\tilde \X_n\leq \tilde \Y_n)$ 
converges to
\[
	P_0 = 1 - 2\int_{-\infty}^0 \frac{e^{-z^2/2}}{\sqrt{2\pi}} \left( \Phi(-z) 
	- 1/2 \right) \, dz = \frac34.
\]
In fact, Theorem~\ref{thm:asymptstochdom} shows that we can obtain any value 
between $0$ and~$1$ for the limit by adding $\floor{K\sqrt n}$ successes 
to~$k_n$, for $K\in\R$.

Next we turn to the proofs of our results. Results in 
Section~\ref{ssec:stochDomination} are proved in 
Section~\ref{sec:stochDomination}, results in 
Section~\ref{ssec:weakConvergence} are proved in 
Section~\ref{sec:weakConvergence} and finally, results in 
Section~\ref{ssec:asymptStochDomination} are proved in 
Section~\ref{sect:asymptStochDomination}.

\section[Domination: finite case]{Stochastic domination of finite vectors}
\label{sec:stochDomination}

Let $\X=(X_1,\dots,X_n)$ and $\Y=(Y_1,\dots,Y_n)$ be vectors of independent 
Bernoulli random variables with success probabilities $p_1,\dots,p_n$ and 
$q_1,\dots,q_n$ respectively, where $0< p_i\leq q_i< 1$ for $i\in 
\{1,\dots,n\}$.

Suppose that $p_i = p$ for all~$i$. Then $\sum_{i=1}^n X_i$ has a binomial 
distribution with parameters $n$ and~$p$. The quotient
\[
	\frac{\Pr(\sum_{i=1}^n X_i = k+1)}{\Pr(\sum_{i=1}^n X_i = k)} = 
	\frac{n-k}{k+1} \frac{p}{1-p}
\]
is strictly increasing in~$p$ and strictly decreasing in~$k$, and it is also 
easy to see that
\[
	\textstyle \Law(\X | \sum_{i=1}^n X_i = k)
	\preceq \Law(\X | \sum_{i=1}^n X_i = k+1).
\]
The following two lemmas show that these two properties hold for general 
success probabilities $p_1,\dots,p_n$.

\begin{lemma}\label{lemma:quotients}
	For $k\in \{0,1,\dots,n-1\}$, consider the quotients
	\begin{equation}\label{eqn:quotienteq}
		Q^n_k := \frac{\Pr(\sum_{i=1}^n X_i = k+1)}{\Pr(\sum_{i=1}^n X_i = k)}
	\end{equation}
	and
	\begin{equation}\label{eqn:quotientgeq}
		\frac{\Pr(\sum_{i=1}^n X_i \geq k+1)}{\Pr(\sum_{i=1}^n X_i \geq k)}.
	\end{equation}
	Both \eqref{eqn:quotienteq} and~\eqref{eqn:quotientgeq} are strictly 
	increasing in $p_1,\dots,p_n$ for fixed~$k$, and strictly decreasing 
	in~$k$ for fixed $p_1,\dots,p_n$.
\end{lemma}

\begin{proof}
	We only give the proof for~\eqref{eqn:quotienteq}, since the proof 
	for~\eqref{eqn:quotientgeq} is similar. First we will prove that $Q^n_k$ 
	is strictly increasing in $p_1,\dots,p_n$ for fixed~$k$. By symmetry, it 
	suffices to show that $Q^n_k$ is strictly increasing in~$p_1$. We show 
	this by induction on~$n$. The base case $n=1$, $k=0$ is immediate. Next 
	note that for $n\geq2$ and $k\in \{0,\dots,n-1\}$,
	\[\begin{split}
		Q^n_k
		&= \frac{\Pr(\sum_{i=1}^{n-1} X_i = k) p_n + \Pr(\sum_{i=1}^{n-1} X_i 
		= k+1) (1-p_n)} {\Pr(\sum_{i=1}^{n-1} X_i = k-1) p_n + 
		\Pr(\sum_{i=1}^{n-1} X_i = k) (1-p_n)} \\
		&= \frac{p_n + Q^{n-1}_k (1-p_n)}{p_n / Q^{n-1}_{k-1} + (1-p_n)},
	\end{split}\]
	which is strictly increasing in~$p_1$ by the induction hypothesis (in the 
	case $k=n-1$, use $Q^{n-1}_k = 0$, and in the case $k=0$, use 
	$1/Q^{n-1}_{k-1} = 0$).

	To prove that $Q^n_k$ is strictly decreasing in~$k$ for fixed 
	$p_1,\dots,p_n$, note that since $Q^n_k$ is strictly increasing in~$p_n$ 
	for fixed $k\in \{1,\dots,n-2\}$, we have
	\[
		0
		< \frac{\partial}{\partial p_n} Q^n_k = \frac{\partial}{\partial p_n} 
		\frac{p_n + Q^{n-1}_k (1-p_n)} {p_n/Q^{n-1}_{k-1} + (1-p_n)}
		= \frac{1 - Q^{n-1}_k / Q^{n-1}_{k-1}}{\bigl( p_n / Q^{n-1}_{k-1} + 
		(1-p_n) \bigr)^2}.
	\]
	Hence, $Q^{n-1}_k < Q^{n-1}_{k-1}$. This argument applies for any 
	$n\geq2$.
\end{proof}

Let $\X^k = (X^k_1,\dots,X^k_n)$ have the conditional law of~$\X$, conditioned 
on the event $\{\sum_{i=1}^n X_i = k\}$. Our next lemma gives an explicit 
coupling of the~$\X^k$ in which they are ordered. The existence of such a 
coupling was already proved in~\cite[Proposition~6.2]{JN}, but our explicit 
construction is new and of independent value. In our construction, we freely 
regard $\X^k$ as a random subset of $\{1,\dots,n\}$ by identifying $\X^k$ with 
$\{i\in \{1,\dots,n\} \colon X^k_i = 1\}$. For any $K\subset \{1,\dots,n\}$, 
let $\{\X_K= \1\}$ denote the event $\{X_i=1\ \forall i\in K\}$, and for any 
$I\subset \{1,\dots,n\}$ and $j\in \{1,\dots,n\}$, define
\[
	\gamma_{j,I}
	:= \sum_{L\subset\{1,\dots,n\}\colon \car{L}=\car{I}+1}
		\frac{\I(j\in L)}{\car{L\setminus I}}
		\Pr( \X_L= \1 \mid {\textstyle \sum_{i=1}^n}X_i = \car{I}+1 ).
\]

\begin{lemma}\label{lem:erik}
	For any $I\subset \{1,\dots,n\}$, the collection $\{\gamma_{j,I}\}_{j\in 
	\lbrace1,\dots,n\rbrace \setminus I}$ is a probability vector. Moreover, 
	if~$I$ is picked according to~$\X^k$ and then~$j$ is picked according to 
	$\{\gamma_{j,I}\}_{j\in \lbrace1,\dots,n\rbrace \setminus I}$, the 
	resulting set $J = \{I,j\}$ has the same distribution as if it was picked 
	according to~$\X^{k+1}$. Therefore, we can couple the sequence 
	$\{\X^k\}_{k=1}^n$ such that $\Pr(\X^1\leq \X^2\leq \dots\leq \X^{n-1}\leq 
	\X^n) = 1$.
\end{lemma}

\begin{proof}
	Throughout the proof, $I$, $J$, $K$ and~$L$ denote subsets of 
	$\{1,\dots,n\}$, and we simplify notation by writing $\Sigma_n := 
	\sum_{i=1}^n X_i$. First observe that
	\[
		\sum_{j\notin I} \gamma_{j,I}
		= \sum_{L\colon \car{L}=\car{I}+1} \Pr(\X_L=\1 \mid \Sigma_n = 
		\car{I}+1) = 1,
	\]
	which proves that the $\{\gamma_{j,I}\}_{j\notin I}$ form a probability 
	vector, since $\gamma_{j,I}\geq0$.

	Next note that for any~$K$ containing~$j$,
	\begin{equation}\label{eqn:couplinglemma}
		\frac{\Pr(\X_K= \1 \mid \Sigma_n=\car{K})}
		{\Pr(\X_{K\setminus\{j\}}= \1 \mid \Sigma_n=\car{K}-1)}
		= \frac{\Pr(X_j=1)}{\Pr(X_j=0)}
		\frac{\Pr(\Sigma_n=\car{K}-1)}{\Pr(\Sigma_n=\car{K})}.
	\end{equation}
	Now fix~$J$, and for $j\in J$, let $I = I(j,J) = J\setminus\{j\}$. Then 
	for $j\in J$, by~\eqref{eqn:couplinglemma},
	\[\begin{split}
		\gamma_{j,I}
		&= \frac{\Pr(\X_J= \1 \mid \Sigma_n=\car{J})}
				{\Pr(\X_I= \1 \mid \Sigma_n=\car{I})}
			\sum_{L\colon \car{L}=\car{J}}
			\frac{\I(j\in L)}{\car{L\setminus I}}
			\Pr(\X_{L\setminus\{j\}}= \1 \mid \Sigma_n=\car{I}) \\
		&= \frac{\Pr(\X_J= \1 \mid \Sigma_n=\car{J})}
				{\Pr(\X_I= \1 \mid \Sigma_n=\car{I})}
			\sum_{K\colon \car{K}=\car{I}}
			\frac{\I(j\notin K)}{\car{J\setminus K}}
			\Pr(\X_K= \1 \mid \Sigma_n=\car{I}),
	\end{split}\]
	where the second equality follows upon writing $K = L\setminus \{j\}$, and 
	using $\car{L\setminus I} = \car{L\setminus J}+1 = \car{K\setminus J}+1 = 
	\car{J\setminus K}$ in the sum. Hence, by summing first over~$j$ and then 
	over~$K$, we obtain
	\[
		\sum_{j\in J} \gamma_{j,I} \Pr(\X_I=\1 \mid \Sigma_n =\car{I})
		= \Pr(\X_J=\1 \mid \Sigma_n=\car{J}). \qedhere
	\]
\end{proof}

\begin{corollary}\label{cor:k,k+1}
	For $k\in\{0,1,\dots,n-1\}$ we have
	\[
		\textstyle \Law(\X | \sum_{i=1}^n X_i \geq k)
		\preceq \Law(\X | \sum_{i=1}^n X_i \geq k+1).
	\]
\end{corollary}

\begin{proof}
	Using Lemma~\ref{lem:erik}, we will construct random vectors $\U$ and~$\V$ 
	on a common probability space such that $\U$ and~$\V$ have the conditional 
	distributions of~$\X$ given $\{\sum_{i=1}^n X_i\geq k\}$ and~$\X$ given 
	$\{\sum_{i=1}^n X_i\geq k+1\}$, respectively, and $\U\leq\V$ with 
	probability~$1$.

	First pick an integer~$m$ according to the conditional law of 
	$\sum_{i=1}^n X_i$ given $\{\sum_{i=1}^n X_i\geq k\}$. If $m\geq k+1$, 
	then pick~$\U$ according to the conditional law of~$\X$ given 
	$\{\sum_{i=1}^n X_i = m\}$, and set $\V=\U$. If $m=k$, then first pick an 
	integer $m+\ell$ according to the conditional law of $\sum_{i=1}^n X_i$ 
	given $\{\sum_{i=1}^n X_i \geq k+1\}$. Next, pick $\U$ and~$\V$ such that 
	$\U$ and~$\V$ have the conditional laws of~$\X$ given $\{\sum_{i=1}^n X_i 
	= m\}$ and~$\X$ given $\{\sum_{i=1}^n X_i = m+\ell\}$, respectively, and 
	$\U\leq\V$. This is possible by Lemma~\ref{lem:erik}. By construction, 
	$\U\leq\V$ with probability~$1$, and a little computation shows that $\U$ 
	and~$\V$ have the desired marginal distributions.
\end{proof}

Now we are in a position to prove Propositions~\ref{prop:sums}, 
\ref{prop:equallaws} and~\ref{prop:stochdom}.

\begin{proof}[Proof of Proposition~\ref{prop:sums}]
	By Lemma~\ref{lemma:quotients} we have that for $\ell\in\{1,\dots,n-k\}$,
	\[
		\frac{\Pr(\sum_{i=1}^n X_i\geq k+\ell)}{\Pr(\sum_{i=1}^n X_i\geq k)}
		= \prod_{j=0}^{\ell-1} \frac{\Pr(\sum_{i=1}^n X_i \geq k+j+1)}
		{\Pr(\sum_{i=1}^n X_i \geq k+j)}
	\]
	is strictly increasing in $p_1,\dots,p_n$. This implies that for $\ell\in 
	\{1,\dots,n-k\}$,
	\[
		\textstyle
		\Pr(\sum_{i=1}^n X_i \geq k+\ell \mid \sum_{i=1}^n X_i \geq k) \leq
		\Pr(\sum_{i=1}^n Y_i \geq k+\ell \mid \sum_{i=1}^n Y_i \geq k).
		\qedhere
	\]
\end{proof}

\begin{proof}[Proof of Proposition~\ref{prop:equallaws}]
	Let $x,y\in \{0,1\}^n$ be such that $\sum_{i=1}^n x_i = \sum_{i=1}^n y_i$  
	and let $k = \sum_{i=1}^n x_i$. Write $I=\{i\in \{1,\dots,n\}\colon x_i = 
	1\}$ and, likewise, $J=\{i\in \{1,\dots,n\}\colon y_i = 1\}$, and recall 
	the definition~\eqref{eqn:defbeta} of~$\beta_i$. We have
	\begin{multline}\label{eqn:equallawseq}
		\frac{\Pr(\X=x\mid \sum_{i=1}^n X_i = k)}{\Pr(\X=y\mid \sum_{i=1}^n 
		X_i = k)} = \frac{\prod_{i\in I} p_i \prod_{i\notin I} (1-p_i)} 
		{\prod_{i\in J} p_i \prod_{i\notin J} (1-p_i)} \\
		= \prod_{i\in I \setminus J} \frac{p_i}{1-p_i} \prod_{i\in J \setminus 
		I} \frac{1-p_i}{p_i}
		= \frac{\prod_{i\in I\setminus J} \beta_i}{\prod_{i\in J\setminus I} 
		\beta_i} \frac{\Pr(\Y=x \mid \sum_{i=1}^n Y_i = k)}{\Pr(\Y = y \mid 
		\sum_{i=1}^n Y_i = k)}.
	\end{multline}

	Since $\car{I} = \car{J} = k$, we have $\car{I\setminus J} = 
	\car{J\setminus I}$. Hence, (i) implies~(ii), and (ii) trivially 
	implies~(iii). To show that (iii) implies~(i), suppose that $\Law(\X | 
	\sum_{i=1}^n X_i = k) = \Law(\Y | \sum_{i=1}^n Y_i = k)$ for a given $k\in 
	\{1,\dots,n-1\}$. Let $i\in \{2,\dots,n\}$ and let~$K$ be a subset of 
	$\{2,\dots,n\}\setminus \{i\}$ with exactly $k-1$ elements. Choosing $I = 
	\{1\}\cup K$ and $J = K\cup\{i\}$ in~\eqref{eqn:equallawseq} yields 
	$\beta_i = \beta_1$.
\end{proof}

\begin{proof}[Proof of Proposition~\ref{prop:stochdom}]
	By Proposition~\ref{prop:equallaws} and Lemma~\ref{lem:erik}, we have for 
	$m\in \{0,1,\dots,n\}$ and $\ell\in \{0,1,\dots,n-m\}$
	\[
		\textstyle \Law(\X | \sum_{i=1}^n X_i = m)
		\preceq \Law(\Y | \sum_{i=1}^n Y_i = m+\ell).
	\]
	Using this result and Proposition~\ref{prop:sums}, we will construct 
	random vectors $\U$ and~$\V$ on a common probability space such that $\U$ 
	and~$\V$ have the conditional distributions of~$\X$ given $\{\sum_{i=1}^n 
	X_i \geq k\}$ and~$\Y$ given $\{\sum_{i=1}^n Y_i \geq k\}$, respectively, 
	and $\U\leq\V$ with probability~$1$.

	First, pick integers $m$ and $m+\ell$ such that they have the conditional 
	laws of $\sum_{i=1}^n X_i$ given $\{\sum_{i=1}^n X_i \geq k\}$ and 
	$\sum_{i=1}^n Y_i$ given $\{\sum_{i=1}^n Y_i \geq k\}$, respectively, and 
	$m\leq m+\ell$ with probability~$1$. Secondly, pick $\U$ and~$\V$ such 
	that they have the conditional laws of~$\X$ given $\{\sum_{i=1}^n X_i = 
	m\}$ and~$\Y$ given $\{\sum_{i=1}^n Y_i = m+\ell\}$, respectively, and 
	$\U\leq\V$ with probability~$1$. A little computation shows that the 
	vectors $\U$ and~$\V$ have the desired marginal distributions.
\end{proof}

We close this section with a minor result, which gives a condition under which 
we do not have stochastic ordering.

\begin{proposition}
	If $p_i=q_i$ for some $i\in\{1,\dots,n\}$ but not for all~$i$, then for 
	$k\in\{1,\dots,n-1\}$,
	\[
		\textstyle \Law(\X | \sum_{i=1}^n X_i \geq k)
		\not\preceq \Law(\Y | \sum_{i=1}^n Y_i \geq k).
	\]
\end{proposition}

\begin{proof}
	Without loss of generality, assume that $p_n=q_n$. We have
	\begin{multline*}
		\textstyle \Pr(X_n=1 \mid \sum_{i=1}^n X_i \geq k) \\
	\begin{aligned}
		&= \rule[-4ex]{0pt}{0pt} \frac{p_n \Pr(\sum_{i=1}^{n-1} X_i \geq 
		k-1)}{p_n \Pr(\sum_{i=1}^{n-1} X_i \geq k-1) + (1-p_n) 
		\Pr(\sum_{i=1}^{n-1} X_i \geq k)} \\
		&= \rule[-4ex]{0pt}{0pt} \frac{p_n}{p_n + (1-p_n) \Pr(\sum_{i=1}^{n-1} 
		X_i \geq k) \big/ \Pr(\sum_{i=1}^{n-1} X_i \geq k-1)} \\
		&> \rule[-5ex]{0pt}{0pt} \frac{q_n}{q_n + (1-q_n) \Pr(\sum_{i=1}^{n-1} 
		Y_i \geq k) \big/ \Pr(\sum_{i=1}^{n-1} Y_i \geq k-1)} \\
		&= \textstyle \Pr(Y_n=1 \mid \sum_{i=1}^n Y_i \geq k),
	\end{aligned}
	\end{multline*}
	where the strict inequality follows from Lemma~\ref{lemma:quotients}.
\end{proof}

\section{Weak convergence}
\label{sec:weakConvergence}

We now turn to the framework for asymptotic domination described in 
Section~\ref{ssec:framework} and to the setting of 
Section~\ref{ssec:weakConvergence}. Recall that $X_{in}$ is the number of 
successes of the vector~$\X_n$ in block~$i$. We want to study the joint 
convergence in distribution of the~$X_{in}$ as $n\to\infty$, conditioned on 
$\{\Sigma_n \geq k_n\}$, and also conditioned on $\{\Sigma_n = k_n\}$. Since 
we are interested in the limit $n\to\infty$, we may assume from the outset 
that the values of~$n$ we consider are so large that $k_n$ and all~$m_{in}$ 
are strictly between $0$ and~$n$, to avoid degenerate situations.

We will first consider the case where we condition on the event $\{\Sigma_n = 
k_n\}$. Lemma~\ref{lem:centres} below states that the~$X_{in}$ will then 
concentrate around the values $c_{in} m_{in}$, where the~$c_{in}$ are 
determined by the system of equations~\eqref{eqn:centres}, which we repeat 
here for the convenience of the reader:
\[
	\left\{ \begin{aligned}
		&\frac{1-c_{in}}{c_{in}} \frac{p_i}{1-p_i}
		= \frac{1-c_{jn}}{c_{jn}} \frac{p_j}{1-p_j}
		&&& \forall i,j\in\{1,\dots,M\}; \\
		& \textstyle\sum_{i=1}^M c_{in} m_{in} = k_n.
	\end{aligned} \right.
	\tag{\ref{eqn:centres}}
\]
Before we turn to the proof of this concentration result, let us first look at 
the system~\eqref{eqn:centres} in more detail. If we write
\begin{equation}\label{eqn:An}
	A_n = \frac{1-c_{in}}{c_{in}} \frac{p_i}{1-p_i}
\end{equation}
for the desired common value for all~$i$, then
\[
	c_{in} = \frac{p_i}{p_i+A_n(1-p_i)}.
\]
Note that this is equal to~$1$ for $A_n = 0$ and to~$p_i$ for $A_n = 1$, and 
strictly decreasing to~$0$ as $A_n\to\infty$, so that there is a unique 
$A_n>0$ such that
\begin{equation}\label{eqn:uniqueAn}
	\sum_{i=1}^M c_{in} m_{in}
	= \sum_{i=1}^M \frac{p_i m_{in}}{p_i+A_n(1-p_i)} = k_n.
\end{equation}
It follows that the system~\eqref{eqn:centres} does have a unique solution, 
characterized by this value of~$A_n$. Moreover, it follows 
from~\eqref{eqn:uniqueAn} that if $k_n > \Ex(\Sigma_n) = \sum_{i=1}^M p_i 
m_{in}$, then $A_n<1$. Furthermore, $k_n/n \to \alpha$ and $m_{in}/n \to 
\alpha_i$. Hence, by dividing both sides in~\eqref{eqn:uniqueAn} by $n$, and 
taking the limit $n\to\infty$, we see that the~$A_n$ converge to the unique 
positive number~$A$ such that
\[
	\sum_{i=1}^M \frac{p_i\alpha_i}{p_i+A(1-p_i)} = \alpha,
\]
where $A=1$ if $\alpha = \sum_{i=1}^M p_i\alpha_i$. As a consequence, we also 
have that
\[
	c_{in} \to c_i = \frac{p_i}{p_i+A(1-p_i)} \qquad \text{as $n\to\infty$}.
\]
Note that the~$c_i$ are the unique solution to the system of equations
\[
	\left\{ \begin{aligned}
		&\frac{1-c_i}{c_i} \frac{p_i}{1-p_i}
		= \frac{1-c_j}{c_j} \frac{p_j}{1-p_j}
		&&& \forall i,j\in\{1,\dots,M\}; \\
		& \textstyle\sum_{i=1}^M c_i \alpha_i = \alpha.
	\end{aligned} \right.
\]
Observe also that $c_i = p_i$ in case $A=1$, or equivalently $\sum_{i=1}^M p_i 
\alpha_i = \alpha$, which is the case when the total number of successes~$k_n$ 
is within~$o(n)$ of the mean~$\Ex(\Sigma_n)$. The concentration result:

\begin{lemma}\label{lem:centres}
	Let $c_{1n},\dots,c_{Mn}$ satisfy~\eqref{eqn:centres}. Then for each~$i$
	and all positive integers~$r$, we have that
	\[
	\Pr(\abs{X_{in}-c_{in}m_{in}} \geq Mr \mid \Sigma_n = k_n)
		\leq 2M e^{-(M-1)r^2/n}.
	\]
\end{lemma}

\begin{proof}
	The idea of the proof is as follows. Condition on $\{\Sigma_n = k_n\}$, 
	and consider the event that for some $i\neq j$ we have that $X_{in} = 
	c_{in}m_{in} + s$, and $X_{jn} = c_{jn} m_{jn} - t$, for some positive 
	numbers $s$ and~$t$. We will show that if the~$c_{in}$ 
	satisfy~\eqref{eqn:centres}, the event obtained by increasing $X_{in}$ 
	by~$1$ and decreasing $X_{jn}$ by~$1$ has smaller probability. This 
	establishes that the conditional distribution of the~$X_{in}$ is maximal 
	at the central values $c_{in}m_{in}$ identified by the 
	system~\eqref{eqn:centres}. The precise bound in Lemma~\ref{lem:centres} 
	also follows from the argument.

	Now for the details. Let $s$ and~$t$ be nonnegative real numbers such that 
	$c_{in}m_{in}+s$ and $c_{jn}m_{jn}-t$ are integers. By the binomial 
	distributions of $X_{in}$ and~$X_{jn}$ and their independence, if it is 
	the case that $0 \leq c_{in}m_{in}+s < m_{in}$ and $0 < c_{jn}m_{jn}-t 
	\leq m_{jn}$, then
	\begin{multline*}
		\frac{\Pr(X_{in} = c_{in}m_{in}+s+1, X_{jn} = c_{jn}m_{jn}-t-1)}
			 {\Pr(X_{in} = c_{in}m_{in}+s, X_{jn} = c_{jn}m_{jn}-t)} \\
		\begin{aligned}
			&= \left( \frac{m_{in}-c_{in}m_{in}-s} {c_{in}m_{in}+s+1} 
			\frac{p_i}{1-p_i} \right) \left( \frac{c_{jn}m_{jn}-t} 
			{m_{jn}-c_{jn}m_{jn}+t+1} \frac{1-p_j}{p_j} \right) \\
			&\leq \left( \frac{m_{in}-c_{in}m_{in}-s} {c_{in}m_{in}} 
			\frac{p_i}{1-p_i} \right) \left( \frac{c_{jn}m_{jn}-t} 
			{m_{jn}-c_{jn}m_{jn}} \frac{1-p_j}{p_j} \right).
		\end{aligned}
	\end{multline*}
	Hence, if the~$c_{in}$ satisfy~\eqref{eqn:centres}, then using $1-z \leq 
	\exp(-z)$ we obtain
	\begin{multline*}
		\frac{\Pr(X_{in} = c_{in}m_{in}+s+1, X_{jn} = c_{jn}m_{jn}-t-1)}
			 {\Pr(X_{in} = c_{in}m_{in}+s, X_{jn} = c_{jn}m_{jn}-t)} \\
		\leq \left( 1-\frac{s}{m_{in}-c_{in}m_{in}} \right) \left( 
		1-\frac{t}{c_{jn}m_{jn}} \right)
		\leq \exp\left( -\frac{s+t}{n} \right).
	\end{multline*}
	It follows by iteration of this inequality, that for \emph{all} real 
	$s,t\geq0$ and all integers~$u\geq0$,
	\begin{multline}\label{eqn:centrebound}
		\Pr(X_{in} = c_{in}m_{in}+s+u, X_{jn} = c_{jn}m_{jn}-t-u) \\
		\leq \exp\left( -\frac{(s+t)u}n \right) \Pr(X_{in} = c_{in}m_{in}+s, 
		X_{jn} = c_{jn}m_{jn}-t).
	\end{multline}

	Now fix~$i$, and observe that for all integers~$r>0$,
	\begin{multline*}
		\Pr(X_{in} \geq c_{in}m_{in} + Mr, \Sigma_n = k_n) \\
		= \sum_{\substack{\ell_1,\dots,\ell_M\in\N_0\colon \\ 
		\ell_1+\dots+\ell_M = k_n}} \!\! \I(\ell_i \geq c_{in}m_{in} + Mr) 
		\Pr(X_{kn} = \ell_k\ \forall k).
	\end{multline*}
	But if $\ell_1+\dots+\ell_M = k_n$ and $\ell_i \geq c_{in}m_{in} + Mr$, 
	then there must be some $j\neq i$ such that $\ell_j \leq c_{jn}m_{jn}-r$. 
	Therefore,
	\begin{multline*}
		\Pr(X_{in} \geq c_{in}m_{in} + Mr, \Sigma_n = k_n) \\
		\leq \sum_{j=1}^M \sum_{\substack{\ell_1,\dots,\ell_M\in\N_0\colon \\
		\ell_1+\dots+\ell_M = k_n}} \!\! \I\left( \substack{\textstyle \ell_i 
		\geq c_{in}m_{in} + Mr\\ \textstyle \ell_j \leq c_{jn}m_{jn}-r} 
		\right) \Pr(X_{kn} = \ell_k\ \forall k).
	\end{multline*}
	By independence of the~$X_{in}$ and using~\eqref{eqn:centrebound} with 
	$s=(M-1)r$, $t=0$ and $u=r$, we now obtain
	\begin{multline*}
		\Pr(X_{in} \geq c_{in}m_{in} + Mr, \Sigma_n = k_n) \\
		\begin{aligned}
			&\leq e^{-(M-1)r^2/n } \sum_{j=1}^M 
			\sum_{\substack{\ell_1,\dots,\ell_M\in\N_0\colon \\ 
			\ell_1+\dots+\ell_M = k_n}} \!\! \I\left( \substack{\textstyle 
			\ell_i \geq c_{in}m_{in} + Mr-r\\ \textstyle \ell_j \leq 
			c_{jn}m_{jn}} \right) \Pr(X_{kn} = \ell_k\ \forall k) \\
			&\leq Me^{-(M-1)r^2/n } \Pr(\Sigma_n = k_n).
		\end{aligned}
	\end{multline*}
	This proves that
	\[
		\Pr(X_{in} \geq c_{in}m_{in} + Mr\mid \Sigma_n = k_n)
		\leq Me^{-(M-1)r^2/n }.
	\]
	Similarly, one can prove that
	\[
		\Pr(X_{in} \leq c_{in}m_{in} - Mr\mid \Sigma_n = k_n)
		\leq Me^{-(M-1)r^2/n }. \qedhere
	\]
\end{proof}

As we have already mentioned, we expect that the~$X_{in}$ have fluctuations 
around their centres of the order~$\sqrt n$. It is therefore natural to look 
at the $M$-dimensional vector
\begin{equation}\label{eqn:defXnproof}
	\Xscaled_n := \left( \frac{X_{1n}-x_{1n}} {\sqrt n}, \frac{X_{2n}-x_{2n}} 
	{\sqrt n}, \dots, \frac{X_{Mn}-x_{Mn}} {\sqrt n} \right),
\end{equation}
where the vector $x_n = (x_{1n},\dots,x_{Mn})$ represents the centre around 
which the~$X_{in}$ concentrate. To prove weak convergence of~$\Xscaled_n$, we 
will not set $x_{in}$ equal to $c_{in} m_{in}$, because the latter numbers are 
not necessarily integer, and it will be more convenient if the~$x_{in}$ are 
integers. So instead, for each fixed~$n$, we choose the~$x_{in}$ to be 
nonnegative integers such that $\abs{x_{in} - c_{in} m_{in}} < 1$ for all~$i$, 
and $\sum_{i=1}^M x_{in} = k_n$. Of course, the vector~$\Xscaled_n$ as it is 
defined in~\eqref{eqn:defXnproof}, and the vector defined 
in~\eqref{eqn:defXnintro} have the same weak limit. In our proofs of Theorems 
\ref{thm:limitgivensuccesses} and~\ref{thm:limitabovemean}, $\Xscaled_n$ will 
refer to the vector defined in~\eqref{eqn:defXnproof}.

If we condition on $\{\Sigma_n = k_n\}$, then the vector~$\Xscaled_n$ will 
only take values in the hyperplane
\[
	S_0 := \{(z_1,\dots,z_M) \in \R^M\colon z_1+\dots+z_M = 0\}.
\]
However, as we have already explained in the introduction, we still 
regard~$\Xscaled_n$ as an $M$-dimensional vector, because we will also 
condition on $\{\Sigma_n \geq k_n\}$, in which case~$\Xscaled_n$ is not 
restricted to a hyperplane. To deal with this, it turns out that for technical 
reasons which will become clear later, it is useful to introduce the 
projection $\pi\colon (z_1,\dots,z_M) \mapsto (z_1,\dots,z_{M-1})$ and the 
shear transformation $\sigma\colon (z_1,\dots,z_M) \mapsto (z_1,\dots,z_{M-1}, 
z_1+\dots+z_M)$. We can then define a metric~$\rho$ on~$\R^M$ by setting 
$\rho(x,y) := \abs{\sigma x-\sigma y}$, where $\abs{\cdot}$ denotes Euclidean 
distance. See Figure~\ref{fig:ShearTransform} for an illustration.

\begin{figure}
	\begin{center}
		\includegraphics[scale=1.2]{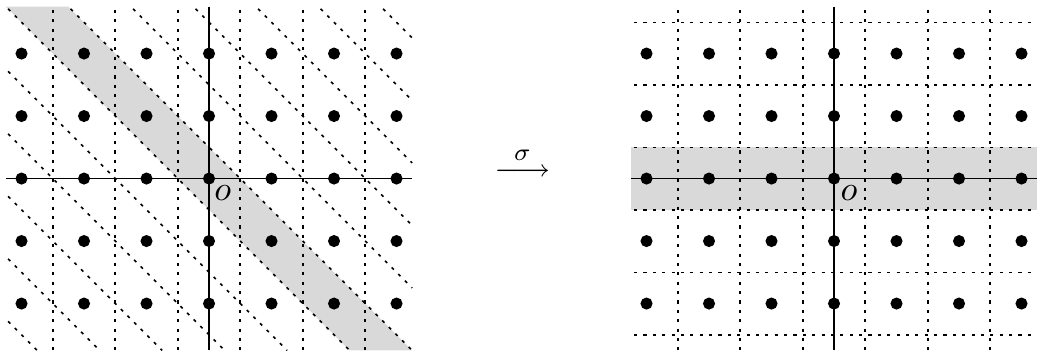}
	\end{center}
	\caption{The shear transformation~$\sigma$ (illustrated here for $M=2$) 
	maps sheared cubes to cubes. The dots are the sites of the integer 
	lattice~$\Z^2$. The gray band on the left encompasses those sheared cubes 
	that intersect~$S_0$.}
	\label{fig:ShearTransform}
\end{figure}

Using the projection~$\pi$, we now define a new measure~$\mu_0$ on the Borel 
subsets of~$\R^M$, which is concentrated on~$S_0$, by
\[
	\mu_0(\,\cdot\,) := \lambda^{M-1}(\pi(\,\cdot\,\cap S_0)),
\]
where $\lambda^{M-1}$ is the ordinary Lebesgue measure on~$\R^{M-1}$. Note 
that up to a multiplicative constant, $\mu_0$ is equal to the measure~$\nu_0$ 
defined in Section~\ref{sec:weakConvergence}, so we could have stated Theorems 
\ref{thm:limitgivensuccesses} and~\ref{thm:limitabovemean} equally well with 
$\mu_0$ instead of~$\nu_0$. In the proofs it turns out to be more convenient 
to work with~$\mu_0$, however, so that is what we shall do.

Our proofs of Theorems \ref{thm:limitgivensuccesses} 
and~\ref{thm:limitabovemean} resemble classical arguments to prove weak 
convergence of random vectors living on a lattice via a local limit theorem 
and Scheff\'e's theorem, see for instance \cite[Theorem~3.3]{billingsley}. 
However, we cannot use these classic results here, for two reasons. First of 
all, in Theorem~\ref{thm:limitabovemean} our random vectors live on an 
$M$-dimensional lattice, but in the limit all the mass collapses onto a 
lower-dimensional hyperplane, leading to a weak limit which is singular with 
respect to $M$-dimensional Lebesgue measure. The classic arguments do not 
cover this case of a singular limit.

Secondly, we are considering \emph{conditioned} random vectors, for which it 
is not so obvious how to obtain a local limit theorem directly. Our solution 
is to get rid of the conditioning by considering \emph{ratios} of conditioned 
probabilities, and prove a local limit theorem for these ratios. An extra 
argument will then be needed to prove weak convergence. Since we cannot resort 
to classic arguments here, we have to go through the proofs in considerable 
detail.

\subsection[Proof Theorem~\ref{thm:limitgivensuccesses}]{Proof of 
Theorem~\ref{thm:limitgivensuccesses}}

As we have explained above, the key idea in the proof of 
Theorem~\ref{thm:limitgivensuccesses} is that we can get rid of the awkward 
conditioning by considering \emph{ratios} of conditional probabilities, rather 
than the conditional probabilities themselves. Thus, we will be dealing with 
ratios of binomial probabilities, and the following lemma addresses the key 
properties of these ratios needed in the proof. The lemma resembles standard 
bounds on binomial probabilities, but we point out that here we are 
considering \emph{ratios} of binomial probabilities which centre around 
$c_{in} m_{in}$ rather than around the mean~$p_i m_{in}$. We also note that 
actually, the lemma is stronger than required to prove 
Theorem~\ref{thm:limitgivensuccesses}, but we will need this stronger result 
to prove Theorem~\ref{thm:limitabovemean} later.

\begin{lemma}\label{lem:binomialratios}
	Recall the definition~\eqref{eqn:An} of~$A_n$. Fix $i\in \{1,2,\dots,M\}$ 
	and let $b_1, b_2, \dotsc$ be a sequence of positive integers such that 
	$b_n/\sqrt{n} \to 0$ as $n\to\infty$. Then, for every $z\in\R$,
	\[
		\sup_{\substack{x\colon \abs{x-x_{in}} < b_n\\ r\colon 
		\abs{r-z\sqrt{n}} < b_n}} \left| \frac1{A_n^r} \frac{\Pr(X_{in} = 
		x+r)}{\Pr(X_{in} = x)} - \exp\left( -\frac{z^2}{2c_i(1-c_i)\alpha_i} 
		\right) \right| \to 0.
	\]
	Furthermore, there exist constants $B^1_i, B^2_i < \infty$ such that for 
	all $n$ and~$r$,
	\[
		\sup_{x\colon \abs{x-x_{in}} < b_n} \frac1{A_n^r} \frac{\Pr(X_{in} = 
		x+r)}{\Pr(X_{in} = x)} \leq B^1_i \left(1 + \frac{r^4}{n^2} \right) 
		\exp\left( B^2_i \frac{\abs{r}}{\sqrt n} - \frac12 \frac{r^2}{n} 
		\right).
	\]
\end{lemma}

\begin{proof}
	Robbins' note on Stirling's formula~\cite{robbins} states that for all $m 
	= 1,2,\dotsc$,
	\[
		\sqrt{2\pi} \, m^{m+1/2} \, e^{-m + 1/(12m+1)} < m! <
		\sqrt{2\pi} \, m^{m+1/2} \, e^{-m + 1/(12m)},
	\]
	from which it is straightforward to show that for all $m = 0,1,2,\dotsc$ 
	(so including $m=0$), there exists an~$\eta_m$ satisfying $1/7<\eta_m<1/5$ 
	such that
	\begin{equation}\label{eqn:stirling}
		m! = \sqrt{2\pi(m+\eta_m)} \, m^m \, e^{-m}
		= \sqrt{2\pi\factrnd{m}} \, m^m \, e^{-m},
	\end{equation}
	where we have introduced the notation $\factrnd{m} := m+\eta_m$.
	
	Since $X_{in}$ has the binomial distribution with parameters $m_{in}$ 
	and~$p_i$,
	\[
		\frac1{A_n^r} \frac{\Pr(X_{in} = x+r)}{\Pr(X_{in} = x)}
		= \frac{x!}{(x+r)!} \, \frac{(m_{in}-x)!}{(m_{in}-x-r)!} \left(
		\frac{c_{in}}{1-c_{in}} \right)^r.
	\]
	Using~\eqref{eqn:stirling}, we can write this as the product of the three 
	factors
	\begin{align*}
		P^1_{in}(x,r) &= \left( \frac{\factrnd{x}} {\factrnd{x+r}} \:
		\frac{\factrnd{m_{in}-x}} {\factrnd{m_{in}-x-r}} \right)^{1/2} \\
		P^2_{in}(x,r) &= \left( \frac{c_{in}m_{in}} {x}
		\frac{m_{in}-x}{m_{in}-c_{in}m_{in}} \right)^r \\
		P^3_{in}(x,r) &= \left( \frac{x} {x+r} \right)^{x+r} \!\! \left(
		\frac{m_{in}-x} {m_{in}-x-r} \right)^{m_{in}-x-r}
	\end{align*}
	for all $x$ and~$r$ such that $0<x<m_{in}$ and $0\leq x+r \leq m_{in}$.

	To study the convergence of $P^3_{in}(x,r)$, first write
	\[
		P^3_{in}(x,r) = \left( 1-\frac{r}{x+r} \right)^{x+r} \!\! \left( 
		1+\frac{r}{m_{in}-x-r} \right)^{m_{in}-x-r}.
	\]
	Using the fact that for all $u>-1$, $(1+u)$ lies between $\exp\bigl( 
	u-\tfrac12 u^2 \bigr)$ and $\exp\bigl( u-\tfrac12 u^2/(1+u) \bigr)$, a 
	little computation now shows that $P^3_{in}(x,r)$ is wedged in between
	\[
		\exp\left( -\frac12 \frac{(m_{in}-r)r^2}{x(m_{in}-x-r)} \right)
		\quad\text{and}\quad
		\exp\left( -\frac12 \frac{(m_{in}+r)r^2}{(x+r)(m_{in}-x)} \right).
	\]
	From this fact, it follows that for fixed $z\in\R$,
	\[
		\sup_{\substack{x\colon \abs{x-x_{in}} < b_n\\ r\colon 
		\abs{r-z\sqrt{n}} < b_n}} \left| P^3_{in}(x,r) - \exp\left( 
		-\frac{z^2} {2c_i(1-c_i)\alpha_i} \right) \right| \to 0,
	\]
	because $x_{in}/m_{in} \to c_i$, hence $x = c_im_{in}+o(n)$ and $r = 
	z\sqrt{n}+o(\sqrt{n})$ under the supremum, and $m_{in}/n \to \alpha_i$. 
	Since $\abs{x_{in} - c_{in}m_{in}} < 1$, we also have that
	\[
		\sup_{\substack{x\colon \abs{x-x_{in}} < b_n\\ r\colon 
		\abs{r-z\sqrt{n}} < b_n}} \left| P^1_{in}(x,r)-1 \right| \to 0
		\quad \text{and} \quad
		\sup_{\substack{x\colon \abs{x-x_{in}} < b_n\\ r\colon 
		\abs{r-z\sqrt{n}} < b_n}} \left| P^2_{in}(x,r)-1 \right| \to 0.
	\]
	Together with the uniform convergence of $P^3_{in}(x,r)$, this establishes 
	the first part of Lemma~\ref{lem:binomialratios}.

	We now turn to the second part of the lemma. If $x$ and~$r$ are such that 
	$0<x<m_{in}$ and $0\leq x+r \leq m_{in}$, then $m_{in}-r \geq x > 0$ and 
	$m_{in}+r \geq m_{in}-x > 0$, hence from the bounds on $P^3_{in}(x,r)$ 
	given in the previous paragraph we can conclude that
	\[
		P^3_{in}(x,r) \leq \exp\left( -\frac12 \frac{r^2}{m_{in}} \right)
		\leq \exp\left( -\frac12 \frac{r^2}{n} \right).
	\]
	Next observe that if $x$ is such that $\abs{x-x_{in}} < b_n$, then $\abs{x 
	- c_{in}m_{in}} < 1+b_n$, from which it follows that uniformly in~$n$, for 
	all~$x$ and~$r$ such that $0 < x < m_{in}$, $0\leq x+ r\leq m_{in}$ and 
	$\abs{x-x_{in}} < b_n$,
	\[
		P^2_{in}(x,r)
		\leq \left(1 + \text{const.} \times \frac{b_n}{n} \right)^{\abs{r}}
		\leq \exp\left( \text{const.} \times \frac{\abs{r}}{\sqrt{n}} \right).
	\]

	To finish the proof, it remains to bound~$P^1_{in}(x,r)$. To this end, 
	observe first that uniformly in~$n$, for all $x$ and~$r$ such that 
	$\abs{x-x_{in}} < b_n$ and $\abs{r} < n^{3/4}$, $P^1_{in}(x,r)$ is bounded 
	by a constant. On the other hand, uniformly for all $x$ and~$r$ such that 
	$0<x<m_{in}$ and $0\leq x+r\leq m_{in}$, $P^1_{in}(x,r)$ is bounded by a 
	constant times~$n$, and $n \leq r^4/n^2$ if $\abs{r} \geq n^{3/4}$. 
	Combining these observations, we see that uniformly in~$n$, for all $x$ 
	and~$r$ satisfying $\abs{x-x_{in}} < b_n$ and $0\leq x+r\leq m_{in}$,
	\[
		P^1_{in}(x,r)
		\leq \text{const.} \times \left(1 + \frac{r^4}{n^2} \right).
		\qedhere
	\]
\end{proof}

\begin{proof}[Proof of Theorem~\ref{thm:limitgivensuccesses}]
	For a point~$z$ in~$\R^M$, let $\round{z}$ be the point in~$\Z^M$ 
	$\rho$-closest to~$z$ (take the lexicographically smallest one if there is 
	a choice). Graphically, this means that the collection of those points~$z$ 
	for which $\round{z} = a$ comprises the sheared cube $a + \sigma^{-1} 
	(-1/2,1/2]^M$, see Figure~\ref{fig:ShearTransform}. Now, for each fixed 
	$z\in \R^M$, set $r^z_n = (r^z_{1n}, \dots, r^z_{Mn}) := \round{z\sqrt 
	n}$. Observe that because (for fixed~$n$) the~$x_{in}$ sum to~$k_n$, if 
	$r^z_n \in S_0$ we have that
	\begin{equation}\label{eqn:ratios}
		\frac{\Pr(\sqrt{n} \, \Xscaled_n = r^z_n \mid \Sigma_n = k_n)} 
		{\Pr(\sqrt{n} \, \Xscaled_n = 0 \mid \Sigma_n = k_n)}
		= \frac{\Pr(\sqrt{n} \, \Xscaled_n = r^z_n)} {\Pr(\sqrt{n} \, 
		\Xscaled_n = 0)}
		= \prod_{i=1}^M \frac{\Pr(X_{in} = x_{in} + r^z_{in})}{\Pr(X_{in} = 
		x_{in})},
	\end{equation}
	where we have used the independence of the components~$X_{in}$. If $r^z_n 
	\notin S_0$, on the other hand, this ratio obviously vanishes.

	We now apply Lemma~\ref{lem:binomialratios} to~\eqref{eqn:ratios}, taking 
	$b_n = M$ for every $n\geq 1$. Since $\sum_{i=1}^M r^z_{in} = 0$ if $r^z_n 
	\in S_0$ and hence $\prod_{i=1}^M A_n^{r^z_{in}} = 1$, the first part of 
	Lemma~\ref{lem:binomialratios} immediately implies that for all $z\in 
	\R^M$,
	\[
		\frac{\Pr(\sqrt{n} \, \Xscaled_n = r^z_n \mid \Sigma_n = k_n)} 
		{\Pr(\sqrt{n} \, \Xscaled_n = 0 \mid \Sigma_n = k_n)}
		\to \I_{S_0}(z) \prod_{i=1}^M \exp\left( -\frac{z_i^2} 
		{2c_i(1-c_i)\alpha_i} \right)
		= f(z)
	\]
	as  $n \to \infty$. To see how this will lead to 
	Theorem~\ref{thm:limitgivensuccesses}, define $f_n\colon \R^M\to\R$ by
	\[
		f_n(z) := (\sqrt n)^M \Pr\bigl( \sqrt{n} \, \Xscaled_n = r^z_n \bigm| 
		\Sigma_n = k_n \bigr).
	\]
	Then~$f_n$ is a probability density function with respect to 
	$M$-dimensional Lebesgue measure~$\lambda$. Moreover, if $\Zscaled_n$ is a 
	random vector with this density, then the vector $\Zscaled'_n = 
	\round{\Zscaled_n\sqrt n} / \sqrt{n}$ has the same distribution as the 
	vector~$\Xscaled_n$, conditioned on $\{\Sigma_n = k_n\}$. Since clearly 
	$\Zscaled_n$ and~$\Zscaled'_n$ must have the same weak limit, it is 
	therefore sufficient to show that the weak limit of~$\Zscaled_n$ has 
	density $f / \int f\,d\mu_0$ with respect to~$\mu_0$.
	
	Now, by what we have established above, we already know that
	\[
		\frac{f_n(z)}{f_n(0)} = \frac{\Pr(\sqrt{n} \, \Xscaled_n = r^z_n \mid 
		\Sigma_n = k_n)} {\Pr(\sqrt{n} \, \Xscaled_n = 0 \mid \Sigma_n = k_n)}
		\to f(z) \qquad \text{for every $z\in\R^M$}.
	\]
	Moreover, the second part of Lemma~\ref{lem:binomialratios} applied 
	to~\eqref{eqn:ratios} shows that the ratios $f_n(z) / f_n(0)$ are 
	uniformly bounded by some $\mu_0$-integrable function~$g(z)$. Thus it 
	follows by dominated convergence that for every Borel set $A\subset \R^M$,
	\[
		\int_A \frac{f_n(z)}{f_n(0)}\, d\mu_0(z) \to \int_A f(z)\, d\mu_0(z).
	\]

	Next observe that $1 = \int f_n\, d\lambda = \int n^{-1/2}f_n\, d\mu_0$, 
	because by the conditioning, $f_n$ is nonzero only on the sheared cubes 
	which intersect~$S_0$. Therefore, taking $A = \R^M$ in the previous 
	equation yields $n^{-1/2} f_n(0) \to (\int f\, d\mu_0)^{-1}$, which in 
	turn implies that for every Borel set~$A$,
	\[
		\int_A n^{-1/2}f_n(z)\, d\mu_0(z) \to \frac{\int_A f(z)\, d\mu_0(z)} 
		{\int f\, d\mu_0}.
	\]
	In general, $\int_F f_n\, d\lambda \neq \int_F n^{-1/2}f_n\, d\mu_0$ for 
	an arbitrary Borel set~$F$, but we have equality here for sufficiently 
	large~$n$ if~$F$ is a finite union of sheared cubes. Hence, if~$A$ is 
	open, we can approximate~$A$ from the inside by unions of sheared cubes 
	contained in~$A$ to conclude that
	\[
		\liminf_{n\to\infty} \int_A f_n(z)\, d\lambda(z) \geq \frac{\int_A 
		f(z)\, d\mu_0(z)} {\int f\, d\mu_0}. \qedhere
	\]
\end{proof}

\subsection[Proof Theorem~\ref{thm:limitabovemean}]{Proof of 
Theorem~\ref{thm:limitabovemean}}

We now turn to the case where we condition on $\{\Sigma_n\geq k_n\}$, for the 
same fixed sequence $k_n\to \infty$ as before. To treat this case, we are 
going to consider what happens when we condition on the event that $\Sigma_n = 
k_n+\ell$ for some $\ell\geq0$, and later sum over~$\ell$. It will be 
important for us to know the relevant range of~$\ell$ to sum over. In 
particular, for large enough~$\ell$ we expect that the probability 
$\Pr(\Sigma_n = k_n+\ell)$ will be so small, that these~$\ell$ will not 
influence the conditional distribution of the vector~$\Xscaled_n$ in an 
essential way. The relevant range of~$\ell$ can be determined from the 
following lemma:

\begin{lemma}\label{lem:rangeofell}
	For all positive integers~$s$,
	\[
		\Pr( \Sigma_n \geq k_n + 2Ms ) \leq M \exp\left( -\frac{ 
		(k_n-\Ex(\Sigma_n)+Ms) s }{Mn} \right) \Pr( \Sigma_n \geq k_n ).
	\]
\end{lemma}

\begin{proof}
	Let $u$ be such that $0<u<(1-p_i)m_{in}$. Observe that then, for all 
	integers~$m$ such that $p_i m_{in} + u \leq m \leq m_{in}$,
	\[
		\frac{\Pr(X_{in} = m+1)}{\Pr(X_{in} = m)}
		= \frac{m_{in}-m}{m+1} \frac{p_i}{1-p_i}
		\leq \frac{p_im_{in}-u\frac{p_i}{1-p_i}}{p_im_{in}+u},
	\]
	hence
	\[
		\frac{\Pr(X_{in} = m+1)}{\Pr(X_{in} = m)}
		\leq 1 - \frac{u}{p_im_{in}+u} \left( 1+\frac{p_i}{1-p_i} \right)
		\leq 1 - \frac{u}{m_{in}} \leq 1 - \frac{u}{n}.
	\]
	Since $1-z \leq \exp(-z)$, by repeated application of this inequality it 
	follows that for all $u>0$ and all positive integers~$t$, if $m$ is an 
	integer such that $m \geq p_i m_{in} + u$, then
	\begin{equation}\label{eqn:rangeofellbound}
		\Pr(X_{in} = m+t)
		\leq \exp\left( -\frac{ut}{n} \right) \Pr(X_{in} = m).
	\end{equation}

	Now observe that if $\Sigma_n \geq \Ex(\Sigma_n) + Mr + 2Ms$, where $s$ is 
	a positive integer, and $r$ a real number such that $r+s > 0$, then for 
	some~$k$ it must be the case that $X_{kn} \geq p_k m_{kn} + r + 2s$.
	Therefore,
	\begin{multline*}
		\Pr( \Sigma_n \geq \Ex(\Sigma_n) + Mr+2Ms ) \\
		\leq \sum_{\substack{ \ell_1,\dots,\ell_M\in \N_0\colon\\ 
		\ell_1+\dots+\ell_M \geq \Ex(\Sigma_n) + Mr+2Ms}} \sum_{k=1}^M 
		\I(\ell_k \geq p_k m_{kn} + r + 2s) \Pr(X_{in} = \ell_i\ \forall i).
	\end{multline*}
	But by~\eqref{eqn:rangeofellbound}, taking $u=r+s$ and $t=s$,
	\begin{multline*}
		\I(\ell_k \geq p_k m_{kn} + r + 2s) \Pr(X_{in} = \ell_i\ \forall i) \\
		\leq \exp\left( -\frac{(r+s)s}{n} \right) \Pr(X_{kn} = \ell_k-s, 
		X_{in} = \ell_i\ \forall i\neq k),
	\end{multline*}
	and therefore
	\begin{multline*}
		\Pr( \Sigma_n \geq \Ex(\Sigma_n) + Mr+2Ms ) \\
		\begin{aligned}
			&\leq M \exp\left( -\frac{(r+s)s}{n} \right) \Pr( \Sigma_n \geq 
			\Ex(\Sigma_n) + Mr+2Ms - s ) \\
			&\leq M \exp\left( -\frac{(r+s)s}{n} \right) \Pr( \Sigma_n \geq 
			\Ex(\Sigma_n) + Mr \bigr).
		\end{aligned}
	\end{multline*}
	Choosing~$r$ such that $k_n \equiv \Ex(\Sigma_n) + Mr$ yields 
	Lemma~\ref{lem:rangeofell} (observe that the bound holds trivially if 
	$r+s\leq0$).
\end{proof}

Lemma~\ref{lem:rangeofell} shows that if $\alpha > \sum_{i=1}^M p_i\alpha_i$, 
then for sufficiently large~$n$, $\Pr(\Sigma_n \geq k_n+\ell)$ will already be 
much smaller than $\Pr(\Sigma_n \geq k_n)$ when $\ell$ is of order $\log n$. 
However, when $\alpha = \sum_{i=1}^M p_i\alpha_i$, we need to consider $\ell$ 
of bigger order than~$\sqrt{n}$ for $\Pr(\Sigma_n \geq k_n+\ell)$ to become 
much smaller than $\Pr(\Sigma_n \geq k_n)$. In either case, 
Lemma~\ref{lem:rangeofell} shows that $\ell$ of larger order than~$\sqrt{n}$ 
become irrelevant.

Keeping this in mind, we will now look at the conditional distribution of the 
vector~$\Xscaled_n$, conditioned on $\{\Sigma_n = k_n + \ell\}$. The first 
thing to observe is that for $\ell>0$, the locations of the centres around 
which the components~$X_{in}$ concentrate will be shifted to larger values. 
Indeed, these centres are located at $c^\ell_{in} m_{in}$, where 
the~$c^\ell_{in}$ are of course determined by the system of equations
\begin{equation}\label{eqn:shiftedcentres}
	\left\{ \begin{aligned}
		&\frac{1-c^\ell_{in}}{c^\ell_{in}} \frac{p_i}{1-p_i}
		= \frac{1-c^\ell_{jn}}{c^\ell_{jn}} \frac{p_j}{1-p_j}
		&&& \forall i,j\in\{1,\dots,M\}; \\
		& \textstyle\sum_{i=1}^M c^\ell_{in} m_{in} = k_n + \ell.
	\end{aligned} \right.
\end{equation}
To find an explicit expression for the size of the shift $c^\ell_{in} - 
c_{in}$, we can substitute $c^\ell_{in} = c_{in} + \delta_{in}$ 
into~\eqref{eqn:shiftedcentres}, and then perform an expansion in powers of 
the correction~$\delta_{in}$ to guess this correction to first order. This 
procedure leads us to believe that~$c^\ell_{in}$ must be of the form
\begin{equation}\label{eqn:shift}
	c^\ell_{in} = c_{in} + c_{in}(1-c_{in}) d^\ell_n + e^\ell_{in},
\end{equation}
where
\[
	d^\ell_n := \frac{\ell}{\sum_{j=1}^M c_{jn}(1-c_{jn})m_{jn}},
\]
and $e^\ell_{in}$ should be a higher-order correction. The following lemma 
shows that the error terms~$e^\ell_{in}$ are indeed of second order 
in~$d^\ell_n$, so that the effective shift in~$c_{in}$ by adding~$\ell$ extra 
successes to our Bernoulli variables is given by $c_{in}(1-c_{in}) d^\ell_n$. 
For convenience, we assume in the lemma that $\abs{d^\ell_n} \leq 1/2$, which 
means that $\abs{\ell}$ cannot be too large, but by 
Lemma~\ref{lem:rangeofell}, this does not put too severe a restriction on the 
range of~$\ell$ we can consider later.

\begin{lemma}\label{lem:shiftedcentres}
	For all~$\ell$ (positive or negative) such that $\abs{d^\ell_n}\leq1/2$, 
	we have that $\abs{e^\ell_{in}} \leq (d^\ell_n)^2$ for all $i=1,\dots,M$.
\end{lemma}

\begin{proof}
	For ease of notation, write $\sigma_{in} := c_{in}(1-c_{in})$. As before, 
	we write
	\[
		A^\ell_n = \frac{1-c^\ell_{in}}{c^\ell_{in}} \frac{p_i}{1-p_i}
		= \frac{1-c_{in}-\sigma_{in}d^\ell_n-e^\ell_{in}}
		{c_{in}+\sigma_{in}d^\ell_n+e^\ell_{in}} \frac{p_i}{1-p_i}
	\]
	for the desired common value for all~$i$, so
	\begin{equation}\label{eqn:shifterror}
		e^\ell_{in} = \frac{p_i(1-c_{in}-\sigma_{in}d^\ell_n) -
		A^\ell_n(1-p_i)(c_{in}+\sigma_{in}d^\ell_n)} {A^\ell_n(1-p_i)+p_i}.
	\end{equation}
	
	As before, the value of~$A^\ell_n$ is uniquely determined by the 
	requirement that $\sum_{i=1}^M c^\ell_{in}m_{in} = k_n+\ell$. Since 
	$\sum_{i=1}^M c_{in}m_{in} = k_n$ and $\sum_{i=1}^M \sigma_{in} d^\ell_n 
	m_{in} = \ell$, this requirement says that
	\[
		\sum_{i=1}^M e^\ell_{in}m_{in} = 0.
	\]
	In particular, the~$e^\ell_{in}$ cannot be all positive or all negative, 
	from which we derive, using~\eqref{eqn:shifterror}, that $A^\ell_n$ must 
	satisfy the double inequalities
	\[
		\min_{i=1,\dots,M} \left\{ \frac{p_i(1-c_{in}-\sigma_{in}d^\ell_n)} 
		{(1-p_i)(c_{in}+\sigma_{in}d^\ell_n)} \right\} \leq A^\ell_n \leq 
		\max_{i=1,\dots,M} \left\{ \frac{p_i(1-c_{in}-\sigma_{in}d^\ell_n)} 
		{(1-p_i)(c_{in}+\sigma_{in}d^\ell_n)} \right\}.
	\]
	A simple calculation establishes that
	\[
		\frac{p_i(1-c_{in}-\sigma_{in}d^\ell_n)}
			{(1-p_i)(c_{in}+\sigma_{in}d^\ell_n)}
		= \frac{1-c_{in}}{c_{in}} \frac{p_i}{1-p_i} \left( 1 + 
		\sum_{k=1}^\infty \frac{(-(1-c_{in})d^\ell_n )^k}{1-c_{in}} \right),
	\]
	from which (using $\abs{d^\ell_n} \leq 1/2$) we can conclude that
	\[
		\frac{1-c_{in}}{c_{in}} \frac{p_i}{1-p_i} \bigl( 1-d^\ell_n \bigr)
		\leq A^\ell_n \leq \frac{1-c_{in}}{c_{in}} \frac{p_i}{1-p_i} 
		(1-d^\ell_n + 2\bigl( d^\ell_n)^2 \bigr),
	\]
	since by~\eqref{eqn:centres}, neither the lower bound nor the upper bound 
	here depends on~$i$.

	Inserting the lower bound on~$A^\ell_n$ into~\eqref{eqn:shifterror} gives
	\[
		e^\ell_{in} \leq \frac{\sigma_{in}(1-c_{in}) (d^\ell_n)^2} 
		{1-(1-c_{in})d^\ell_n} \leq \frac12 (d^\ell_n)^2,
	\]
	where in the last step we used that $\abs{d^\ell_n} \leq 1/2$ and 
	$\sigma_{in} \leq 1/4$. Likewise, substituting the upper bound 
	on~$A^\ell_n$ into~\eqref{eqn:shifterror} yields
	\[
		e^\ell_{in} \geq -\frac{\sigma_{in}(1+c_{in}) (d^\ell_n)^2 + 
		2\sigma_{in}(1-c_{in}) (d^\ell_n)^3} {1 - (1-c_{in})d^\ell_n + 
		2(1-c_{in})(d^\ell_n)^2} \geq -\frac{2\sigma_{in} (d^\ell_n)^2}{1-1/2} 
		\geq -(d^\ell_n)^2. \qedhere
	\]
\end{proof}

For future use, we state the following corollary:

\begin{corollary}\label{cor:cin-ci}
	If $(k_n - \sum_{i=1}^M c_i m_{in}) / \sqrt{n} \to K$ for some $K\in 
	[-\infty,\infty]$, then for $i\in \{1,\dots,M\}$,
	\[
		\frac{(c_{in}-c_i) m_{in}}{\sqrt n}
		\to \frac{c_i(1-c_i)\alpha_i}{\sum_{j=1}^M c_j(1-c_j)\alpha_j} K.
	\]
\end{corollary}

\begin{remark}
	If $(k_n-\Ex(\Sigma_n)) / \sqrt n\to K\in\R$, then $\alpha = \sum_{i=1}^M 
	p_i\alpha_i$ and we have $c_i = p_i$ for all $i\in \{1,\dots,M\}$. In this 
	situation, Corollary~\ref{cor:cin-ci} states that the vectors 
	$\Xscaled^p_n - \Xscaled_n$, and hence also the same vectors conditioned 
	on $\{\Sigma_n\geq k_n\}$, converge pointwise to the vector whose $i$-th 
	component is
	\[
		\frac{p_i(1-p_i)\alpha_i}{\sum_{j=1}^M p_j(1-p_j)\alpha_j} K.
	\]
\end{remark}

\begin{proof}[Proof of Corollary~\ref{cor:cin-ci}]
	First, suppose that $K\in\R$. If $\ell = \sum_{i=1}^M c_i m_{in} - k_n$ 
	and the~$c_{in}^\ell$ satisfy~\eqref{eqn:shiftedcentres}, then 
	$c_{in}^\ell = c_i$. Hence, by Lemma~\ref{lem:shiftedcentres},
	\[
		c_i-c_{in} = c_{in}(1-c_{in})d_n^\ell + O\bigl( (d_n^\ell)^2 \bigr),
	\]
	where
	\[
		d_n^\ell  = \frac{\sum_{i=1}^M c_i m_{in} - k_n}{\sum_{j=1}^M c_{jn} 
		(1-c_{jn}) m_{jn}} = O\bigl( n^{-1/2} \bigr).
	\]
	This implies
	\[
		\frac{(c_i-c_{in})m_{in}}{\sqrt n}
		= \frac{c_{in}(1-c_{in})m_{in}}{\sum_{j=1}^M c_{jn}(1-c_{jn})m_{jn}} 
		\frac{\sum_{i=1}^M c_i m_{in}-k_n}{\sqrt n}+ O\bigl( n^{-1/2} \bigr),
	\]
	from which the result follows.

	Next, suppose that $K=\infty$. Since $c_{in}$ is increasing as a function 
	of~$k_n$, we have by the first part of the proof
	\[
		\liminf_{n\to\infty} \frac{(c_{in}-c_i)m_{in}}{\sqrt n}
		\geq \frac{c_i(1-c_i)\alpha_i}{\sum_{j=1}^M c_j(1-c_j)\alpha_j} L
	\]
	for all $L\in\R$. Hence, the left-hand side is equal to~$\infty$. The 
	proof for the case $K=-\infty$ is similar.
\end{proof}

When we condition on $\{\Sigma_n = k_n+\ell\}$, then in analogy with what we 
have done before, the natural scaled vector to consider would be the vector
\[
	\Xscaled^{\ell}_n := \left( \frac{X_{1n} - x^\ell_{1n}}{\sqrt n}, 
	\frac{X_{2n} - x^\ell_{2n}}{\sqrt n}, \dots, \frac{X_{Mn} - 
	x^\ell_{Mn}}{\sqrt n} \right),
\]
where the components of the vector $x^\ell_n = (x^\ell_{1n}, \dots, 
x^\ell_{Mn})$ identify the centres around which the~$X_{in}$ concentrate. 
Here, the~$x^\ell_{in}$ are nonnegative integers chosen such that 
$\abs{x^\ell_{in} - c^\ell_{in} m_{in}} < 1$ for all~$i$, and $\sum_{i=1}^M 
x^\ell_{in} = k_n+\ell$. Note that the vector~$\Xscaled^{\ell}_n$ is simply a 
translation of~$\Xscaled_n$ by $(x^\ell_n - x_n) / \sqrt{n}$. Since 
Lemma~\ref{lem:rangeofell} shows that if~$k_n$ is sufficiently larger 
than~$\Ex(\Sigma_n)$, only values of~$\ell$ up to small order in~$n$ are 
relevant, the statement of Theorem~\ref{thm:limitabovemean} should not come as 
a surprise. To prove it, we need to refine the arguments we used to prove 
Theorem~\ref{thm:limitgivensuccesses}.

\begin{proof}[Proof of Theorem~\ref{thm:limitabovemean}]
	Assume that $(k_n-\Ex(\Sigma_n)) / \sqrt{n} \to \infty$, and let
	\[
		a_n := 2M \left\lfloor \sqrt{n} \left( \frac{\sqrt n}{k_n - 
		\Ex(\Sigma_n)} \right)^{1/2} \right\rfloor.
	\]
	Note that then $a_n\to\infty$ but $a_n / \sqrt n\to 0$. Furthermore, 
	Lemma~\ref{lem:rangeofell} and a short computation show that
	\[
		\frac{\Pr(\Sigma_n > k_n+a_n)}{\Pr(\Sigma_n \geq k_n)} \to 0.
	\]
	It is easy to see that from this last fact it follows that
	\[
		\sup_A \Bigl| \Pr(\Xscaled_n\in A\mid \Sigma_n\geq k_n) - 
		\Pr(\Xscaled_n\in A\mid k_n\leq \Sigma_n\leq k_n+a_n) \Bigr| \to 0,
	\]
	where the supremum is over all Borel subsets~$A$ of~$\R^M$. It is 
	therefore sufficient to consider the limiting distribution of the 
	vector~$\Xscaled_n$ conditioned on the event $\{k_n\leq\Sigma_n\leq 
	k_n+a_n\}$, rather than on the event $\{\Sigma_n\geq k_n\}$.

	As in the proof of Theorem~\ref{thm:limitgivensuccesses}, for $z\in \R^M$ 
	we let $r^z_n = \round{z\sqrt n}$, and we define the functions $f_n\colon 
	\R^M\to \R$ by setting
	\[
		f_n(z) := (\sqrt{n})^M \Pr\bigl( \sqrt{n} \, \Xscaled_n = r^z_n \bigm| 
		k_n \leq \Sigma_n \leq k_n+a_n \bigr).
	\]
	As before, this is a probability density function with respect to Lebesgue 
	measure~$\lambda$ on~$\R^M$, and if $\Zscaled_n$ is a random vector with 
	this density, then the vector $\Zscaled_n' = \round{\Zscaled_n\sqrt n} / 
	\sqrt{n}$ has the same distribution as the vector~$\Xscaled_n$ conditioned 
	on the event $\{k_n\leq \Sigma_n\leq k_n+a_n\}$. Hence, it is enough to 
	show that the weak limit of~$\Zscaled_n$ has density $f / \int f\, d\mu_0$ 
	with respect to~$\mu_0$.
	
	An essential difference compared to the situation in 
	Theorem~\ref{thm:limitgivensuccesses}, however, is that the 
	densities~$f_n$ are no longer supported by the collection of points~$z$ 
	for which $r^z_n$ is in the hyperplane~$S_0$ (i.e.\ the union of those 
	sheared cubes that intersect~$S_0$). Rather, the support now encompasses 
	all the points~$z$ for which~$r^z_n$ is in any of the hyperplanes
	\[
		S_\ell := \{(z_1,\dots,z_M) \in \R^M\colon z_1+\dots+z_M = \ell\}, 
		\qquad \ell = 0,1,\dots,a_n,
	\]
	because if $r^z_n \in S_\ell$, then the event $\{\sqrt{n} \, \Xscaled_n = 
	r^z_n\}$ is contained in the event $\{\Sigma_n = k_n + \ell\}$. For this 
	reason, the densities~$f_n$ are not so convenient to work with here. 
	Instead, it is more convenient to ``coarse-grain'' our densities by 
	spreading the mass over sheared cubes of volume $( (2a_n+1)/\sqrt n )^M$ 
	rather than volume~$(1/\sqrt n)^M$, to the effect that all the mass is 
	again contained in the collection of sheared (coarse-grained) cubes 
	intersecting~$S_0$.
	
	\begin{figure}
		\begin{center}
			\includegraphics{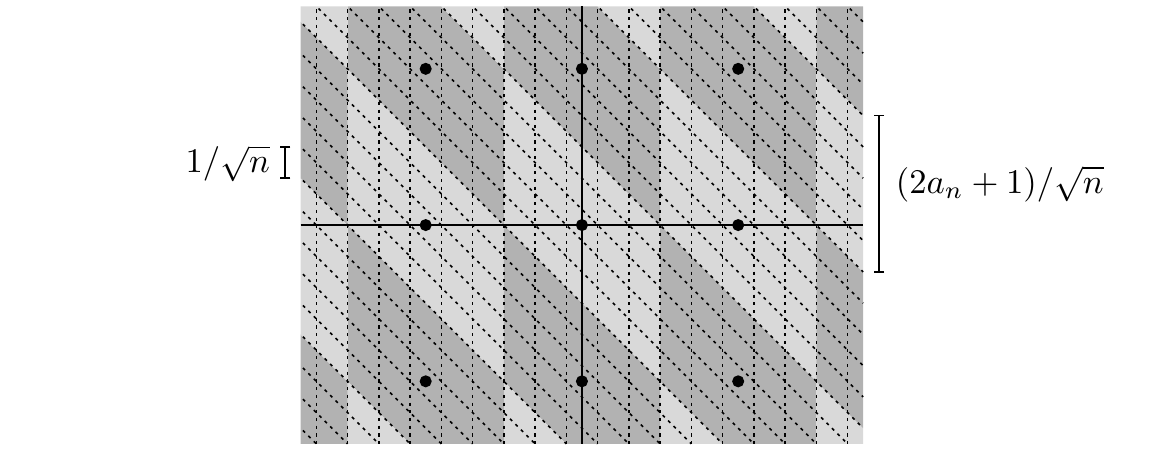}
		\end{center}
		\caption{We coarse-grain our densities by combining $(2a_n+1)^M$ 
		sheared cubes into larger sheared cubes. Here, we show this 
		coarse-graining for $M=2$ and $a_n = 2$. The dots are the points in 
		$((2a_n+1)\Z)^M / \sqrt{n}$. The combined sheared cubes have been 
		coloured in a chessboard fashion as a visual aid.}
		\label{fig:CoarseGraining}
	\end{figure}

	To this end, for given~$n$ we partition $\R^M$ into the collection of 
	sets
	\begin{equation}\label{eqn:coarsepartition}
		\Bigl\{ \frac1{\sqrt n} \bigl( a + \sigma^{-1} (-a_n-1/2, a_n+1/2]^M 
		\bigr)\colon a\in \bigl( (2a_n+1)\Z \bigr)^M \Bigr\}.
	\end{equation}
	See Figure~\ref{fig:CoarseGraining}. For a given point $z\in \R^M$, we 
	denote by~$Q_n^z$ the sheared cube in this partition containing~$z$. Now 
	we can define the coarse-grained densities
	\[\begin{split}
		g_n(z)
		&:= \left( \frac{\sqrt n}{2a_n+1} \right)^{\!\!M} \Pr( \Xscaled_n \in 
		Q^z_n \mid k_n\leq \Sigma_n \leq k_n+a_n ) \\
		&\phantom:= \left( \frac{\sqrt n}{2a_n+1} \right)^{\!\!M} \int_{Q_n^z} 
		f_n(y)\, d\lambda(y).
	\end{split}\]
	By construction, these are again probability density functions with 
	respect to $M$-dimensional Lebesgue measure~$\lambda$. Moreover, each of 
	these densities is supported on the collection of sheared cubes 
	in~\eqref{eqn:coarsepartition} that intersect~$S_0$, and is constant on 
	each sheared cube~$Q^z_n$. In particular, for any given point $z\in\R^M$ 
	we have
	\[
		\int_{Q^z_n} g_n(y)\, d\lambda(y) = \frac{2a_n+1}{\sqrt n} 
		\int_{Q^z_n} g_n(y)\, d\mu_0(y).
	\]
	Finally, because $a_n / \sqrt n\to 0$ it is clear that if~$\Zscaled_n''$ 
	has density~$g_n$, then its weak limit will coincide with that 
	of~$\Zscaled_n$, and hence also with that of the vector~$\Xscaled_n$ 
	conditioned on the event $\{k_n\leq \Sigma_n\leq k_n+a_n\}$.

	Suppose now that we could prove that
	\begin{equation}\label{eqn:coarsegrainedlimit}
		\frac{2a_n+1}{\sqrt n} g_n(z) \to \frac{f(z)}{\int f\, d\mu_0}
		\qquad \text{for every $z\in\R^M$.}
	\end{equation}
	Then it would follow from Fatou's lemma that for every open set $A\subset 
	\R^M$,
	\[
		\liminf_{n\to\infty} \int_A \frac{2a_n+1}{\sqrt n} g_n(z)\, d\mu_0(z)
		\geq \frac{\int_A f(z)\, d\mu_0(z)}{\int f\, d\mu_0}.
	\]
	By approximating the open set~$A$ by unions of sheared cubes contained 
	in~$A$, as in the proof of Theorem~\ref{thm:limitgivensuccesses}, it is 
	then clear that this would imply that
	\[
		\liminf_{n\to\infty} \int_A g_n(z)\, d\lambda(z)
		\geq \frac{\int_A f(z)\, d\mu_0(z)}{\int f\, d\mu_0}.
	\]
	It therefore only remains to establish~\eqref{eqn:coarsegrainedlimit}.

	Since \eqref{eqn:coarsegrainedlimit} holds by construction for $z\notin 
	S_0$, we only need to consider the case $z\in S_0$. So let us fix $z\in 
	S_0$, and look at~$g_n(z)$. By definition, this is just the rescaled 
	conditional probability that the vector~$\Xscaled_n$ lies in the sheared 
	cube~$Q^z_n$, given that $k_n \leq \Sigma_n \leq k_n+a_n$. In other words, 
	if we define $C_n^z := \sqrt{n} Q^z_n \cap \Z^M$ and $C_{\ell n}^z := 
	C_n^z \cap S_\ell$, then we have
	\[\begin{split}
		g_n(z)
		&= \left(\frac{\sqrt n}{2a_n+1}\right)^{\!\!M} \sum_{r\in C_n^z} \Pr( 
		\sqrt{n} \, \Xscaled_n = r \mid k_n \leq \Sigma_n \leq k_n+a_n) \\
		&= \left(\frac{\sqrt n}{2a_n+1}\right)^{\!\!M} \sum_{\ell=0}^{a_n} 
		\sum_{r\in C_{\ell n}^z} \frac{\Pr(\sqrt{n} \, \Xscaled_n = r \mid 
		\Sigma_n = k_n+\ell) \Pr(\Sigma_n = k_n+\ell)}{\Pr(k_n \leq \Sigma_n 
		\leq k_n+a_n)}.
	\end{split}\]
	Since $C^z_{\ell n}$ contains exactly $(2a_n+1)^{M-1}$ points, from this 
	equality we conclude that to prove~\eqref{eqn:coarsegrainedlimit}, it is 
	sufficient to show that
	\begin{equation}\label{eqn:coarseuniformconvergence}
		\sup_{0 \leq \ell \leq a_n} \sup_{r\in C^z_{\ell n}} \left| (\sqrt 
		n)^{M-1} \Pr(\sqrt{n} \, \Xscaled_n = r \mid \Sigma_n = k_n+\ell) - 
		\frac{f(z)} {\int f\, d\mu_0} \right| \to 0.
	\end{equation}

	The proof of~\eqref{eqn:coarseuniformconvergence} proceeds along the same 
	line as the proof of pointwise convergence in 
	Theorem~\ref{thm:limitgivensuccesses}, based on 
	Lemma~\ref{lem:binomialratios}. However, there is a catch: because we are 
	now conditioning on $\Sigma_n = k_n+\ell$, the~$X_{in}$ are no longer 
	centred around~$x_{in}$, but around $x^\ell_{in}$. We therefore first 
	write the conditional probabilities in a form analogous to what we had 
	before, by using that
	\[
		\Pr\bigl( \sqrt{n} \, \Xscaled_n = r \bigm| \Sigma_n = k_n+\ell \bigr) 
		= \Pr\bigl( \sqrt{n} \, \Xscaled^{\ell}_n = r+x_n-x^\ell_n \bigm| 
		\Sigma_n = k_n+\ell \bigr).
	\]
	Writing $r^\ell := r+x_n-x^\ell_n$ for convenience, we now want to study 
	the ratios
	\[
		\frac{\Pr(\sqrt{n} \, \Xscaled^{\ell}_n = r^\ell \mid \Sigma_n = 
		k_n+\ell)} {\Pr(\sqrt{n} \, \Xscaled^{\ell}_n = 0 \mid \Sigma_n = 
		k_n+\ell)}
		= \frac{\Pr(\sqrt{n} \, \Xscaled^{\ell}_n = r^\ell)}
		{\Pr(\sqrt{n} \, \Xscaled^{\ell}_n = 0)}
		= \prod_{i=1}^M \frac{\Pr(X_{in} = x^\ell_{in} +r^\ell_i)}
		{\Pr(X_{in} = x^\ell_{in})}
	\]
	for $\ell$ and~$r$ satisfying $0\leq \ell\leq a_n$ and $r\in C^z_{\ell 
	n}$.

	By equation~\eqref{eqn:shift} and Lemma~\ref{lem:shiftedcentres} we have 
	that $\sup_\ell \abs{x^\ell_{in} - x_{in}} = o(\sqrt{n})$, from which it 
	follows that also $\sup_{\ell,r} \abs{r^\ell-z\sqrt{n}} = o(\sqrt{n})$, 
	where the suprema are over all $\ell\in \{0,\dots,a_n\}$ and $r\in 
	C^z_{\ell n}$. Thus, by the first part of Lemma~\ref{lem:binomialratios},
	\[
		\sup_{0\leq\ell\leq a_n} \sup_{r\in C^z_{\ell n}} \left| 
		\frac{\Pr(\sqrt{n} \, \Xscaled^{\ell}_n = r^\ell \mid \Sigma_n = 
		k_n+\ell)} {\Pr(\sqrt{n} \, \Xscaled^{\ell}_n = 0 \mid \Sigma_n = 
		k_n+\ell)} - f(z) \right| \to 0,
	\]
	where we have used that for all terms concerned, $\prod_{i=1}^M 
	A_n^{r^\ell_i} = 1$ because $r^\ell\in S_0$. Furthermore, from the second 
	part of Lemma~\ref{lem:binomialratios} it follows that the functions
	\[
		z\mapsto \frac{\Pr(\sqrt{n} \, \Xscaled^{\ell}_n = \round{z\sqrt n} 
		\mid \Sigma_n = k_n+\ell )} {\Pr(\sqrt{n} \, \Xscaled^{\ell}_n = 0 
		\mid \Sigma_n = k_n+\ell)}
	\]
	are bounded uniformly in~$n$ and in all $\ell\in \{0, \dots, a_n\}$ by a 
	$\mu_0$-integrable function. In the same way as in the proof of 
	Theorem~\ref{thm:limitgivensuccesses}, it follows from these facts (with 
	the addition that we have uniform bounds) that
	\[
		\sup_{0\leq\ell\leq a_n} \left| (\sqrt n)^{M-1} \Pr(\sqrt{n} \, 
		\Xscaled^{\ell}_n = 0 \mid \Sigma_n = k_n+\ell) - \frac1{\int f\, 
		d\mu_0} \right| \to 0.
	\]
	From this we conclude that~\eqref{eqn:coarseuniformconvergence} does hold, 
	which completes the proof of Theorem~\ref{thm:limitabovemean}.
\end{proof}

\subsection[Proof Theorem~\ref{thm:limitaroundmean}]{Proof of 
Theorem~\ref{thm:limitaroundmean}}

\begin{proof}[Proof of Theorem~\ref{thm:limitaroundmean}]
	Suppose that $(k_n - \Ex(\Sigma_n)) / \sqrt{n} \to K$ for some 
	$K\in[-\infty,\infty)$. Let $\Xscaled$ be a random vector having a 
	multivariate normal distribution with density $h / \int h\,d\lambda$ with 
	respect to~$\lambda$. By standard arguments, $\Xscaled^p_n$ converges 
	weakly to~$\Xscaled$. Therefore, for a rectangle $A\subset \R^M$ we have
	\[
		\Pr(\Xscaled^p_n\in A, \Sigma_n\geq k_n)
		= \Pr(\Xscaled^p_n\in A\cap H_{\frac{k_n-\Ex(\Sigma_n)}{\sqrt n}})
		\to \Pr(\Xscaled\in A\cap H_K),
	\]
	since $A\cap H_{K+\eps}$ is a $\lambda$-continuity set for all 
	$\eps\in\R$. Taking $A=\R^M$ gives
	\[
		\Pr(\Sigma_n\geq k_n) \to \Pr(\Xscaled\in H_K).
	\]
	Hence, for all rectangles $A\subset \R^M$
	\[
		\Pr(\Xscaled^p_n\in A \mid \Sigma_n \geq k_n)
		\to \frac{\Pr(\Xscaled\in A\cap H_K)}{\Pr(\Xscaled\in H_K)}.
		\qedhere
	\]
\end{proof}

\subsection{Law of large numbers}

Finally, we prove a law of large numbers, which we will need in 
Section~\ref{sect:asymptStochDomination}. Let $\tilde X_{in}$ denote a random 
variable with the conditional law of~$X_{in}$, conditioned on the event 
$\{\Sigma_n \geq k_n\}$. If $(k_n-\Ex(\Sigma_n))/\sqrt{n} \to K$ for some 
$K\in[-\infty,\infty]$, then an immediate consequence of Theorems 
\ref{thm:limitabovemean} and~\ref{thm:limitaroundmean} is that $\tilde 
X_{in}/n$ converges in probability to either $p_i\alpha_i$ or $c_i\alpha_i$. 
The following theorem shows that such a law of large numbers holds for a 
general sequence~$k_n$ such that $k_n/n \to \alpha$.

\begin{theorem}\label{thm:lln}
	For $i\in\{1,\dots,M\}$, the random variable~$\tilde X_{in}/n$ converges 
	in probability to $p_i\alpha_i$ if $\alpha \leq \sum_{i=1}^M p_i\alpha_i$, 
	or to $c_i\alpha_i$ if $\alpha\geq \sum_{i=1}^M p_i\alpha_i$.
\end{theorem}

\begin{proof}
	If $\alpha \neq \sum_{i=1}^M p_i\alpha_i$, then $(k_n - \Ex(\Sigma_n)) / 
	\sqrt{n}$ goes to $-\infty$ or~$\infty$ as $n\to\infty$, and the result 
	immediately follows from Theorem~\ref{thm:limitabovemean} and 
	Theorem~\ref{thm:limitaroundmean}.

	Now suppose that $\alpha = \sum_{i=1}^M p_i\alpha_i$. Then $c_i = p_i$ for 
	all $i\in\{1,\dots,M\}$. Recall that in general the~$c_i$ and~$A$ are 
	determined by the equations
	\[
		c_i = \frac{p_i}{p_i+A(1-p_i)} \quad\text{and}\quad
		\sum_{i=1}^M \frac{p_i\alpha_i}{p_i+A(1-p_i)} = \alpha.
	\]
	The constant~$A$ is continuous as a function of~$\alpha$, hence $c_i = 
	c_i[\alpha]$ is also continuous as a function of~$\alpha$. Therefore, if 
	$\alpha = \sum_{i=1}^M p_i\alpha_i$, then for each $\eps>0$ we can choose 
	$\delta>0$ such that $c_i[\alpha+\delta] \alpha_i \leq p_i\alpha_i + 
	\frac{1}{2} \eps$. By Corollary~\ref{cor:k,k+1} we have, for large 
	enough~$n$,
	\begin{multline*}
		\Pr(X_{in}\geq (p_i\alpha_i+\eps)n \mid \Sigma_n\geq k_n) \\
	\begin{aligned}
		&\leq \Pr(X_{in}\geq (p_i\alpha_i+\eps)n \mid \Sigma_n\geq (\alpha + 
		\delta)n) \\
		&\leq \Pr(X_{in}\geq (c_i[\alpha+\delta]\alpha_i + \tfrac{1}{2}\eps) n \mid 
		\Sigma_n\geq (\alpha + \delta)n ),
	\end{aligned}
	\end{multline*}
	which tends to~$0$ as $n\to\infty$ by Theorem~\ref{thm:limitabovemean}. 
	Similarly, using Corollary~\ref{cor:k,k+1} and 
	Theorem~\ref{thm:limitaroundmean} instead of 
	Theorem~\ref{thm:limitabovemean}, we obtain
	\[
		\Pr(X_{in}\leq (p_i\alpha_i-\eps)n \mid \Sigma_n\geq k_n) \to 0.
	\]
	We conclude that $\tilde X_{in}/n$ converges in probability to 
	$p_i\alpha_i = c_i\alpha_i$.
\end{proof}

\section[Asymptotic domination]{Asymptotic stochastic domination}
\label{sect:asymptStochDomination}

\subsection[Proof Theorem~\ref{thm:asymptstochdom}]{Proof of 
Theorem~\ref{thm:asymptstochdom}}

Consider the general framework for vectors $\X_n$ and~$\Y_n$ of 
Section~\ref{ssec:framework} in the setting of 
Section~\ref{ssec:asymptStochDomination}. We will split the proof of 
Theorem~\ref{thm:asymptstochdom} into four lemmas. In the statements of these 
lemmas, we will need the constant~$\hat\alpha$, which is defined as the limit 
as $n\to\infty$ of $\hat k_n/n$:
\[
	\hat k_n = \sum_{i=1}^M \frac{p_i m_{in}}{p_i+\beta_{\max}(1-p_i)}\,,
	\quad\text{hence}\quad
	\hat\alpha = \sum_{i=1}^M \frac{p_i \alpha_i}{p_i+\beta_{\max}(1-p_i)}\,.
\]

Let us first look at the definition of~$\hat \alpha$ in more detail. In 
Section~\ref{ssec:asymptStochDomination}, we informally introduced the 
sequence~$\hat k_n$ as a critical sequence such that if $k_n$ is around~$\hat 
k_n$, then there exists a block~$i$ such that the number of successes~$\tilde 
X_{in}$ of the vector~$\tilde \X_n$ in block~$i$ is roughly the same 
as~$\tilde Y_{in}$. We will now make this precise. Recall that the~$c_i$ and 
the constant~$A$ are determined by
\[
	c_i = \frac{p_i}{p_i+A(1-p_i)}
	\quad\text{and}\quad
	\sum_{i=1}^M \frac{p_i\alpha_i}{p_i+A(1-p_i)} = \alpha.
\]
Furthermore, note that
\[
	\frac{p_i}{p_i+\beta_i(1-p_i)} = q_i,
\]
and recall that we defined $I = \{i\in \{1,\dots,M\}\colon \beta_i = 
\beta_{\max}\}$. The ordering of $\alpha$ and~$\hat \alpha$ gives information 
about the ordering of the~$c_i$ and~$q_i$. This is stated in the following 
remark, which follows from the equations above.

\begin{remark}\label{remark:hatalpha}
	We have the following:
	\begin{itemize}
		\item[(i)] If $\alpha < \hat \alpha$, then $A>\beta_{\max}$ and $c_i < 
			q_i$ for all $i\in\{1,\dots,M\}$.
		\item[(ii)] If $\alpha = \hat \alpha$, then $A = \beta_{\max}$ and 
			$c_i = q_i$ for $i\in I$, while $c_i < q_i$ for $i\notin I$.
		\item[(iii)] If $\alpha > \hat \alpha$, then $A<\beta_{\max}$ and $c_i 
			> q_i$ for some $i\in\{1,\dots,M\}$.
		\item[(iv)] $\sum_{i=1}^M p_i\alpha_i \leq \hat\alpha \leq 
			\sum_{i=1}^M q_i\alpha_i$, with $\hat\alpha = \sum_{i=1}^M 
			p_i\alpha_i$ if and only if $\beta_{\max} = 1$, and $\hat\alpha =  
			\sum_{i=1}^M q_i\alpha_i$ if and only if all $\beta_i$ ($i\in 
			\{1,\dots,M\}$) are equal.
	\end{itemize}
\end{remark}

Our law of large numbers, Theorem~\ref{thm:lln}, states that $\tilde X_{in}/n$ 
converges in probability to $p_i\alpha_i$ if $\alpha \leq \sum_{i=1}^M 
p_i\alpha_i$, and to $c_i\alpha_i$ if $\alpha \geq \sum_{i=1}^M p_i\alpha_i$.    
This law of large numbers applies analogously to the vector~$\tilde\Y_n$. If 
we define $d_1,\dots,d_M$ as the unique solution of the system
\[
	\left\{ \begin{aligned}
		&\frac{1-d_i}{d_i} \frac{q_i}{1-q_i}
		= \frac{1-d_j}{d_j} \frac{q_j}{1-q_j}
		&&& \forall i,j\in\{1,\dots,M\}, \\
		& \textstyle\sum_{i=1}^M d_i \alpha_i = \alpha,
	\end{aligned} \right.
\]
then $\tilde Y_{in}/n$ converges in probability to $q_i\alpha_i$ if $\alpha 
\leq \sum_{i=1}^M q_i\alpha_i$, and to $d_i\alpha_i$ if $\alpha \geq 
\sum_{i=1}^M q_i\alpha_i$. These laws of large numbers and the observations in 
Remark~\ref{remark:hatalpha} will play a crucial role in the proofs in this 
section.

Now we define one-dimensional (possibly degenerate) distribution functions 
$F_K\colon \R\to[0,1]$ for $K\in[-\infty,\infty]$, which will come up in the 
proofs as the distribution functions of the limit of a certain function of the 
vectors~$\tilde \X_n$. Recall from Section~\ref{ssec:weakConvergence} the 
definitions \eqref{eqn:defnu0}, \eqref{eqn:deff}, \eqref{eqn:defh} 
and~\eqref{eqn:defHK} of the measure~$\nu_0$, the functions $f$ and~$h$ and 
the half-space~$H_K$. Write $u = (u_1,\dots,u_M)$. Then
\begin{equation}\label{eqn:defFK}
	F_K(z) = \begin{cases}
		\rule[-4ex]{0pt}{0pt}\hfil \displaystyle \frac{\int_{H_K\cap \{ 
		\sum_{i\in I} u_i\leq z \}} h(u)\,d\lambda(u)}{\int_{H_K} h 
		\,d\lambda}
		& \text{if $K<\infty$, $\alpha = \sum_{i=1}^M p_i\alpha_i$},\\
		\rule[-4ex]{0pt}{0pt}\hfil \displaystyle \frac{\int_{\{ \sum_{i\in I} 
		u_i\leq z-z_K \}} f(u) \,d\nu_0(u)}{\int f\,d\nu_0}
		& \text{if $K<\infty$, $\alpha > \sum_{i=1}^M p_i\alpha_i$}, \\
		\hfil 0 & \text{if $K=\infty$}, \\
	\end{cases}
\end{equation}
where
\begin{equation}\label{eqn:defzK}
	z_K = \frac{\sum_{i\in I} c_i(1-c_i)\alpha_i}{\sum_{i=1}^M c_i 
	(1-c_i)\alpha_i} K.
\end{equation}

The following lemmas, together with Proposition~\ref{prop:stochdom}, imply 
Theorem~\ref{thm:asymptstochdom}.

\begin{lemma}\label{lem:belowhatalpha}
	If $\alpha<\hat\alpha$, then $\sup \Pr(\tilde\X_n \leq \tilde\Y_n)\to 1$.
\end{lemma}

\begin{lemma}\label{lem:abovehatalpha}
	Suppose that $\alpha>\hat\alpha$ and $\beta_i\neq\beta_j$ for some $i,j\in 
	\{1,\dots,M\}$. Then $\sup \Pr(\tilde\X_n \leq \tilde\Y_n)\to 0$.
\end{lemma}

\begin{lemma}\label{lem:aroundhatalpha}
	Suppose that $\alpha = \hat\alpha$ and $\beta_i\neq\beta_j$ for some 
	$i,j\in \{1,\dots,M\}$. Suppose furthermore that $(k_n - \hat k_n) / 
	\sqrt{n}\to K$ for some $K\in [-\infty,\infty]$. Then $\sup \Pr(\tilde\X_n 
	\leq \tilde\Y_n)\to \inf_{z\in\R} F_K(z) - \Phi(z/a) + 1$.
\end{lemma}

\begin{lemma}\label{lem:PK}
	If $\alpha = \hat\alpha$ and $\beta_i\neq\beta_j$ for some $i,j\in 
	\{1,\dots,M\}$, then
	\[
		\inf_{z\in\R} F_K(z)-\Phi(z/a)+1 = \begin{cases}
			\hfil1 & \text{if $K=-\infty$}, \\
			P_K & \text{if $K\in\R$}, \qquad \text{where $0<P_K<1$}, \\
			\hfil0 & \text{if $K=\infty$}.
		\end{cases}
	\]
\end{lemma}

The constant~$a$ in Lemma~\ref{lem:aroundhatalpha} is the constant defined 
in~\eqref{eqn:defa}. The infimum in Lemma~\ref{lem:aroundhatalpha} can 
actually be computed, as Lemma~\ref{lem:PK} states, and attains the values 
stated in Theorem~\ref{thm:asymptstochdom}, with $P_K$ as defined 
in~\eqref{eqn:defPK}.

We will prove Theorem~\ref{thm:asymptstochdom} by proving each of the Lemmas 
\ref{lem:belowhatalpha}--\ref{lem:PK} in turn. The idea behind the proof of 
Lemma~\ref{lem:belowhatalpha} is as follows. If we do not condition at all, 
then $\X_n\preceq \Y_n$ for every $n\geq 1$. If $\alpha < \sum_{i=1}^M p_i 
\alpha_i$, then the effect of conditioning vanishes in the limit and 
$\sup\Pr(\tilde \X_n\leq \tilde \Y_n) \to 1$ as $n\to\infty$. If $\sum_{i=1}^M 
p_i \alpha_i \leq \alpha < \hat \alpha$, then $c_i<q_i$ for all 
$i\in\{1,\dots,M\}$. Hence, for large~$n$ we have that $\tilde X_{in}$ is 
significantly smaller than~$\tilde Y_{in}$ for all $i\in\{1,\dots,M\}$, from 
which it will again follow that $\sup\Pr(\tilde \X_n \leq \tilde \Y_n)\to 1$.

\begin{proof}[Proof of Lemma \ref{lem:belowhatalpha}]
	First, suppose that $\alpha < \sum_{i=1}^M p_i \alpha_i$. Let $\X_n$ 
	and~$\Y_n$ be defined on a common probability space $(\Omega, \mathcal{F}, 
	P)$ such that $\X_n\leq \Y_n$ on all of~$\Omega$. Pick $\omega_1\in 
	\Omega$ according to the measure $P(\,\cdot\mid \sum_{i=1}^M X_{in}\geq 
	k_n)$ and pick $\omega_2\in \Omega$ independently according to the measure 
	$P(\,\cdot\mid \sum_{i=1}^M Y_{in}\geq k_n)$. If $\omega_2$ is in the 
	event $\bigl\{ \sum_{i=1}^M X_{in}\geq k_n \bigr\}\in \mathcal{F}$, set 
	$\tilde \Y_n(\omega_1, \omega_2) := \Y_n(\omega_1)$, otherwise set $\tilde 
	\Y_n(\omega_1, \omega_2) := \Y_n(\omega_2)$. Set $\tilde \X_n(\omega_1, 
	\omega_2) := \X_n(\omega_1)$ regardless of the value of~$\omega_2$. It is 
	easy to see that this defines a coupling of $\tilde \X_n$ and~$\tilde 
	\Y_n$ on the space $(\Omega\times\Omega, \mathcal{F}\times \mathcal{F})$ 
	with the correct marginals for $\tilde \X_n$ and~$\tilde \Y_n$. Moreover, 
	in this coupling we have $\tilde \X_n\leq \tilde \Y_n$ at least if 
	$\omega_2 \in \bigl\{ \sum_{i=1}^M X_{in}\geq k_n \bigr\}$. Hence
	\[
		\sup \Pr(\tilde\X_n \leq \tilde\Y_n)
		\geq \frac{\Pr(\sum_{i=1}^M X_{in} \geq k_n)}
				{\Pr(\sum_{i=1}^M Y_{in}\geq k_n)},
	\]
	which tends to~$1$ as $n\to \infty$ (e.g. by Chebyshev's inequality).

	Secondly, suppose that $\sum_{i=1}^M p_i \alpha_i \leq \alpha < \hat 
	\alpha$. By Remark~\ref{remark:hatalpha}(i), $c_i < q_i$ for all 
	$i\in\{1,\dots,M\}$. For each coupling of $\tilde \X_n$ and~$\tilde \Y_n$ 
	we have
	\[
		\Pr(\tilde\X_n \leq \tilde\Y_n) \geq \Pr(\tilde X_{in}\leq 
		(c_i+q_i)\alpha_i n/2\leq \tilde Y_{in}\ \forall i\in\{1,\dots,M\}),
	\]
	which tends to~$1$ as $n\to\infty$ by Theorem~\ref{thm:lln} and 
	Remark~\ref{remark:hatalpha}(iv).
\end{proof}

The next lemma, Lemma~\ref{lem:abovehatalpha}, treats the case $\alpha > \hat 
\alpha$. In this case, we have that for large~$n$, $\tilde X_{in}$ is 
significantly larger than~$\tilde Y_{in}$ for some $i\in \{1,\dots,M\}$, from 
which it follows that  $\sup\Pr(\tilde \X_n\leq \tilde \Y_n)\to 0$.

\begin{proof}[Proof of Lemma~\ref{lem:abovehatalpha}]
	First, suppose that $\hat \alpha < \alpha < \sum_{i=1}^M q_i \alpha_i$. 
	Then $c_i > q_i$ for some $i\in \{1,\dots,M\}$ by 
	Remark~\ref{remark:hatalpha}(iii). Hence, by Theorem~\ref{thm:lln} and 
	Remark~\ref{remark:hatalpha}(iv),
	\begin{align*}
		\Pr(\tilde X_{in}\geq (c_i+q_i)\alpha_i n/2) &\to 1, \\
		\Pr(\tilde Y_{in}\geq (c_i+q_i)\alpha_i n/2) &\to 0.
	\end{align*}
	It follows that $\Pr(\tilde\X_n \leq \tilde\Y_n)$ tends to~$0$ uniformly 
	over all couplings.

	Next, suppose that $\alpha \geq \sum_{i=1}^M q_i\alpha_i$ and $\beta_i 
	\neq \beta_j$ for some $i,j\in \{1,\dots,M\}$. Then there exists 
	$i\in\{1,\dots,M\}$ such that $c_i\neq d_i$, since
	\[
		\frac{1-d_i}{d_i} \frac{d_j}{1-d_j} \beta_j
		= \frac{1-q_i}{q_i} \frac{p_j}{1-p_j}
		= \beta_i \frac{1-c_i}{c_i} \frac{c_j}{1-c_j}.
	\]
	In fact, we must have $c_i > d_i$ for some $i\in\{1,\dots,M\}$, because
	$\sum_{i=1}^M c_i \alpha_i = \sum_{i=1}^M d_i \alpha_i$. By 
	Theorem~\ref{thm:lln}, it follows that
	\begin{align*}
		\Pr(\tilde X_{in}\geq (c_i+d_i) \alpha_i n / 2) &\to 1, \\
		\Pr(\tilde Y_{in}\geq (c_i+d_i) \alpha_i n / 2) &\to 0.
	\end{align*}
	Again, $\Pr(\tilde \X_n\leq \tilde \Y_n)$ tends to~$0$ uniformly over all 
	couplings.
\end{proof}

We now turn to the proof of Lemma~\ref{lem:aroundhatalpha}. Under the 
assumptions of this lemma, $c_i = q_i$ for $i\in I$ and $c_i < q_i$ for 
$i\notin I$. The proof proceeds in four steps. In step~1, we show that the 
blocks~$i\notin I$ do not influence the asymptotic behaviour of 
$\sup\Pr(\tilde \X_n\leq \tilde \Y_n)$, because for these blocks, $\tilde 
X_{in}$ is significantly smaller than~$\tilde Y_{in}$ for large~$n$. In 
step~2, we show that the parts of the vectors $\tilde \X_n$ and~$\tilde \Y_n$ 
that correspond to the blocks~$i\in I$ are stochastically ordered, if and only 
if the total numbers of successes in these parts of the vectors are 
stochastically ordered. At this stage, the original problem of stochastic 
ordering of random vectors has been reduced to a problem of stochastic 
ordering of random variables. In step~3, we use our central limit theorems to 
deduce the asymptotic behaviour of the total numbers of successes in the 
blocks~$i\in I$. In step~4, we apply the following lemma, which follows 
from~\cite[Proposition~1]{ruschendorf}, to these total numbers of successes:

\begin{lemma}\label{lemma:couplingprobability}
	Let $X$ and~$Y$ be random variables with distribution functions $F$ 
	and~$G$ respectively. Then we have
	\[
		\sup \Pr(X\leq Y) = \inf_{z\in\R} F(z) - G(z) + 1,
	\]
	where the supremum is taken over all possible couplings of $X$ and~$Y$.
\end{lemma}

\begin{proof}[Proof of Lemma \ref{lem:aroundhatalpha}]
	Write $m_{In} := \sum_{i\in I} m_{in}$. Let $\X_{In}$ and~$\tilde \X_{In}$ 
	denote the $m_{In}$-dimensional subvectors of $\X_n$ and~$\tilde \X_n$, 
	respectively, consisting of the components that belong to  the 
	blocks~$i\in I$. Define $\Y_{In}$ and~$\tilde\Y_{In}$ analogously.

	\textbf{Step~1}. Note that for each coupling of $\tilde \X_n$ and~$\tilde 
	\Y_n$,
	\begin{multline}\label{eqn:sumerror}
		\begin{aligned}
		\Pr(\tilde \X_n \leq \tilde \Y_n)
			&\geq \Pr(\tilde\X_{In} \leq\tilde \Y_{In}, \tilde X_{in}\leq (c_i 
			+ q_i)\alpha_i n/2\leq \tilde Y_{in}\ \forall i\notin I) \\
			&\geq \Pr(\tilde\X_{In} \leq \tilde\Y_{In}) - \null
		\end{aligned}\\
			\sum_{i\notin I} \Bigl\{ \Pr\Bigl( \tilde X_{in} > \frac{c_i + 
			q_i}2 \alpha_i n \Bigr) + \Pr\Bigl( \tilde Y_{in} < \frac{c_i + 
			q_i}2 \alpha_i n \Bigr) \Bigr\}.
	\end{multline}
	By Remark~\ref{remark:hatalpha}(ii), $c_i < q_i$ for $i\notin I$. Hence, 
	it follows from Remark~\ref{remark:hatalpha}(iv) and Theorem~\ref{thm:lln} 
	that the sum in~\eqref{eqn:sumerror} tends to~$0$ as $n\to\infty$, 
	uniformly over all couplings. Since clearly $\sup\Pr(\tilde \X_n \leq 
	\tilde \Y_n) \leq \sup\Pr(\tilde \X_{In} \leq \tilde \Y_{In})$,
	\[
		\left| \sup \Pr(\tilde \X_n \leq \tilde \Y_n) - \sup \Pr(\tilde 
		\X_{In} \leq \tilde \Y_{In}) \right| \to 0,
	\]
	where the suprema are taken over all possible couplings of $(\tilde 
	\X_n,\tilde \Y_n)$ and $(\tilde \X_{In},\tilde \Y_{In})$, respectively.

	\textbf{Step~2}. The~$\beta_i$ for $i\in I$ are all equal. Hence, by 
	Proposition~\ref{prop:equallaws} and Lemma~\ref{lem:erik} we have for 
	$m\in \{0,1,\dots,m_{In}\}$ and $\ell\in \{0,1,\dots,m_{In}-m\}$
	\begin{equation}\label{eqn:XorXtilde}
		\textstyle \Law(\X_{In} | \sum_{i\in I} X_{in} = m) \preceq 
		\Law(\Y_{In} | \sum_{i\in I} Y_{in} = m+\ell).
	\end{equation}
	Now let $B$ be any collection of vectors of length~$m_{In}$ with exactly 
	$m$~components equal to~$1$ and $m_{In}-m$ components equal to~$0$. Then
	\[\begin{split}
		\Pr(\tilde\X_{In}\in B)
		&= {\textstyle \Pr(\X_{In}\in B\mid \sum_{i=1}^M X_{in}\geq k_n)} \\
		&= \frac{\Pr(\X_{In}\in B) \Pr(\sum_{i\notin I} X_{in}\geq k_n-m)}
			{\Pr(\sum_{i=1}^M X_{in}\geq k_n)}.
	\end{split}\]
	Taking $C$ to be the collection of all vectors in $\{0,1\}^{m_{In}}$ with  
	exactly $m$~components equal to~$1$, we obtain
	\[
		{\textstyle\Pr(\tilde\X_{In}\in B\mid\sum_{i\in I}\!\tilde X_{in}=m)}
		= \frac{\Pr(\tilde\X_{In}\in B)}{\Pr(\tilde\X_{In}\in C)}
		= {\textstyle \Pr(\X_{In}\in B\mid \sum_{i\in I}\!X_{in} = m)},
	\]
	and likewise for $\Y_{In}$ and~$\tilde\Y_{In}$. Hence, 
	\eqref{eqn:XorXtilde} is equivalent to
	\[
		\textstyle \Law(\tilde \X_{In} | \sum_{i \in I} \tilde X_{in} = m) 
		\preceq \Law(\tilde \Y_{In} | \sum_{i\in I} \tilde Y_{in} = m+\ell).
	\]
	With a similar argument as in the proof of 
	Proposition~\ref{prop:stochdom}, it follows that
	\[
		\textstyle \sup \Pr(\tilde \X_{In} \leq \tilde \Y_{In}) = \sup 
		\Pr(\sum_{i\in I} \tilde X_{in} \leq \sum_{i\in I} \tilde Y_{in}).
	\]

	\textbf{Step~3}. First observe that by Remark~\ref{remark:hatalpha}(iv), 
	$\alpha < \sum_{i=1}^M q_i\alpha_i$. Hence, by 
	Theorem~\ref{thm:limitaroundmean} (note that $(k_n-\Ex(\sum_{i=1}^M 
	Y_{in}))/\sqrt n\to -\infty$) and the continuous mapping theorem,
	\begin{equation}\label{eqn:limitY}
		\textstyle \Pr(\sum_{i\in I}(\tilde Y_{in}-q_im_{in}) / \sqrt{n} \leq 
		z) \to \Phi(z/a) \quad \text{for every $z\in\R$}.
	\end{equation}

	Next observe that by Remark~\ref{remark:hatalpha}(ii), $c_i = q_i$ for 
	$i\in I$ and $A = \beta_{\max}$, from which it follows that $\hat k_n = 
	\sum_{i=1}^M c_im_{in}$. Hence, Corollary~\ref{cor:cin-ci} gives
	\begin{equation}\label{eqn:limitXshift}
		\textstyle \sum_{i\in I} (c_{in} - q_i) m_{in} / \sqrt{n} \to z_K,
	\end{equation}
	with $z_K$ as defined in~\eqref{eqn:defzK}. In the case $\alpha > 
	\sum_{i=1}^M p_i\alpha_i$, Theorem~\ref{thm:limitabovemean}, 
	\eqref{eqn:limitXshift} and the continuous mapping theorem now immediately 
	imply
	\begin{equation}\label{eqn:limitX}
		\textstyle \Pr(\sum_{i\in I}(\tilde X_{in}-q_im_{in}) / \sqrt{n} \leq 
		z) \to F_K(z) \quad \text{for every $z\in\R$}.
	\end{equation}
	Note that if $K=\pm\infty$, $F_K$ is degenerate in this case: we have 
	$F_K(z)=1$ for all $z\in\R$ if $K=-\infty$ and $F_K(z)=0$ for all $z\in\R$ 
	if $K=\infty$.

	Now consider the case $\alpha = \sum_{i=1}^M p_i\alpha_i$. By 
	Remark~\ref{remark:hatalpha}(iv), in this case we have $\beta_{\max}=1$, 
	which implies that $\hat k_n = \sum_{i=1}^M p_im_{in} = \Ex(\Sigma_n)$ and 
	$p_i = q_i$ for all $i\in \{1,\dots,M\}$. Hence, if $K=\infty$, then 
	\eqref{eqn:limitXshift} and Theorem~\ref{thm:limitabovemean} again 
	imply~\eqref{eqn:limitX} with $F_K(z)=0$ everywhere. If $K\in 
	[-\infty,\infty)$, then we obtain~\eqref{eqn:limitX} directly from 
	Theorem~\ref{thm:limitaroundmean}; $F_K$ is non-degenerate in this case 
	(also for $K=-\infty$).

	\textbf{Step~4}. The distribution functions on the left-hand sides of 
	\eqref{eqn:limitY} and~\eqref{eqn:limitX} are non-decreasing and bounded 
	between $0$ and~$1$, hence they converge uniformly on compact sets. It 
	follows by Lemma~\ref{lemma:couplingprobability} that
	\[
		\textstyle \sup \Pr(\sum_{i\in I} \tilde X_{in} \leq \sum_{i\in I} 
		\tilde Y_{in}) \to \inf_{z\in\R} F_K(z) - \Phi(z/a) + 1. \qedhere
	\]
\end{proof}

Finally, we turn to the proof of Lemma~\ref{lem:PK}. The key to computing the 
infimum of $F_K(z)-\Phi(z/a)+1$ is to first express the distribution 
function~$F_K$, defined  in~\eqref{eqn:defFK}, in a simpler form.

\begin{proof}[Proof of Lemma~\ref{lem:PK}]
	In the case $\alpha > \sum_{i=1}^M p_i\alpha_i$ and $K=-\infty$, $F_K$ 
	is~$1$ everywhere, hence $\inf_{z\in\R} F_K(z)-\Phi(z/a) + 1 = 1$. In the 
	case $K=\infty$, $F_K$ is~$0$ everywhere, hence $\inf_{z\in\R} F_K(z) - 
	\Phi(z/a) + 1 = 0$. We will now study the remaining cases.
	
	Consider the case $\alpha = \hat\alpha = \sum_{i=1}^M p_i \alpha_i$ and 
	$K\in [-\infty,\infty)$. Let $\vect{Z} = (Z_1,\dots,Z_M)$ be a random 
	vector which has the multivariate normal distribution with density $h / 
	\int h\,d\lambda$. By Remark~\ref{remark:hatalpha}(iv) we have 
	$\beta_{\max} = 1$. Note that therefore, $\tfrac1a \sum_{i\in I}Z_i$, 
	$\tfrac1b \sum_{i\notin I}Z_i$ and $\tfrac1c \sum_{i=1}^M Z_i$, with $a$, 
	$b$ and~$c$ as defined in~\eqref{eqn:defabc}, all have the standard normal 
	distribution. Moreover, $\sum_{i\in I}Z_i$ and $\sum_{i\notin I}Z_i$ are 
	independent.
	
	For $K=-\infty$, it follows that $F_K(z) = \Phi(z/a)$, hence 
	$\inf_{z\in\R} F_K(z) - \Phi(z/a) + 1 = 1$. For $K\in\R$, observe that 
	$\vect{Z}\in H_K$ is equivalent with $\tfrac1c \sum_{i=1}^M Z_i \geq K/c$. 
	Likewise, $\vect{Z} \in H_K\cap \{u\in\R^M\colon \sum_{i\in I} u_i\leq 
	z\}$ is equivalent with $\tfrac1a \sum_{i\in I}Z_i \leq z/a$ and $\tfrac1b 
	\sum_{i\notin I}Z_i \geq (K - \sum_{i\in I}Z_i)/b$. It follows that
	\[\begin{split}
		F_K(z)
		&= \frac{\int h\,d\lambda}{\int_{H_K} h\,d\lambda} 
		\frac{\int_{H_K\cap\{\sum_{i\in I}u_i\leq z\}} h(u)\,d\lambda(u)}
		{\int h\,d\lambda} \\
		&= \frac1{1-\Phi(K/c)} \int_{-\infty}^{z/a} \int_{\frac{K-a 
		u}{b}}^\infty \frac{e^{-u^2 / 2}}{\sqrt{2\pi}} 
		\frac{e^{-v^2/2}}{\sqrt{2\pi}} \,dv\,du \\
		&= \int_{-\infty}^{z/a} \frac{e^{-u^2/2}}{\sqrt{2\pi}} \frac{1 - 
		\Phi\left( \frac{K-a u}{b} \right)} {1-\Phi\left(\frac{K}{c}\right)} 
		\,du,
	\end{split}\]
	hence
	\begin{equation}\label{eqn:PK1}
		F_K(z) - \Phi(z/a)
		= \int_{-\infty}^{z/a} \frac{e^{-u^2/2}}{\sqrt{2\pi}} \frac{\Phi\left( 
		\frac{K}{c} \right) - \Phi\left( \frac{K-a u}{b} \right)} 
		{1-\Phi\left(\frac{K}{c}\right)} \,du.
	\end{equation}
	Clearly, the derivative of this expression with respect to~$z$ is~$0$ if 
	and only if $(K-z)/b = K/c$, that is, $z = z_{\min} = K-bK/c$. Plugging 
	this value for~$z$ into~\eqref{eqn:PK1} shows that $\inf_{z\in\R} F_K(z) -  
	\Phi(z/a) + 1 = P_K$, with $P_K$ as defined in~\eqref{eqn:defPK}. 
	Moreover, $P_K>0$ because $F_K(z_{\min})>0$, and $P_K<1$ because the 
	integrand in~\eqref{eqn:PK1} is negative for $u < z_{\min}/a$.

	Finally, consider the case $\alpha > \sum_{i=1}^M p_i \alpha_i$ and 
	$K\in\R$. This time, let $\vect{Z} = (Z_1,\dots,Z_M)$ be a random vector 
	which has the singular multivariate normal distribution with density 
	$f/\int f\,d\nu_0$ with respect to~$\nu_0$. Then a little computation 
	shows that $(Z_1, \dots, Z_{M-1})$ has a multivariate normal distribution 
	with mean~$0$ and a covariance matrix~$\Sigma$ given by
	\[
		\left\{ \begin{aligned}
			\rule[-5ex]{0pt}{0pt}
			\Sigma_{ii} &= \displaystyle \frac{\sigma_i^2 \sum_{k=1, k\neq 
			i}^M \sigma_k^2} {\sum_{k=1}^M \sigma_k^2}
			&&&& \text{for $i \in \{1,\dots,M-1\}$}, \\
			\Sigma_{ij} &= \displaystyle \frac{-\sigma_i^2 
			\sigma_j^2}{\sum_{k=1}^M \sigma_k^2}
			&&&& \text{for $i,j \in \{1,\dots,M-1\}$ with $i\neq j$},
		\end{aligned} \right.
	\]
	where $\sigma_i^2 = c_i(1-c_i)\alpha_i$ for $i\in \{1,\dots,M\}$. 
	Similarly, every subvector of~$\vect{Z}$ of dimension less than~$M$ has a 
	multivariate normal distribution.

	By the definition~\eqref{eqn:defFK} of~$F_K$, $z_K + \sum_{i\in I} Z_i$ 
	has distribution function~$F_K$. Since $\beta_i\neq \beta_j$ for some $i,j 
	\in\{1,\dots,M\}$, we have $\car{I}\leq M-1$. It follows that $\sum_{i\in 
	I} Z_i$ has a normal distribution with mean~$0$ and variance
	\begin{equation}\label{eqn:variance}
		\sum_{i\in I} \frac{\sigma_i^2 \sum_{k=1, k\neq i}^M \sigma_k^2} 
		{\sum_{k=1}^M \sigma_k^2} +
		\sum_{i\in I} \sum_{j\in I\setminus\{i\}} \frac{-\sigma_i^2 
		\sigma_j^2}{\sum_{k=1}^M \sigma_k^2} =
		\frac{(\sum_{i\in I} \sigma_i^2)(\sum_{i\notin I} 
		\sigma_i^2)}{\sum_{i=1}^M \sigma_i^2}.
	\end{equation}
	By Remark~\ref{remark:hatalpha}(ii), $A=\beta_{\max}$ and hence for 
	$i\in\{1,\dots,M\}$,
	\[
		\sigma_i^2 = c_i (1-c_i) \alpha_i = \frac{\beta_{\max} 
		p_i(1-p_i)\alpha_i}{(p_i + \beta_{\max} (1-p_i))^2}.
	\]
	It follows that the variance~\eqref{eqn:variance} is equal to $a^2b^2 / 
	c^2$, with $a$, $b$, and~$c$ as defined in~\eqref{eqn:defabc}. 
	Furthermore, $z_K = a^2K/c^2$. We conclude that $F_K$ is the distribution 
	function of a normally distributed random variable with mean $a^2K/c^2$ 
	and variance $a^2b^2/c^2$, so that $F_K(z) = \Phi\bigl( \frac{c}{ab} ( 
	z-a^2K/c^2 ) \bigr)$. Since $a^2 b^2 / c^2 < a^2$, we see that $F_K(z) < 
	\Phi(z/a)$ for small enough~$z$. Hence $F_K(z) - \Phi(z/a)$ attains a 
	minimum value which is strictly smaller than~$0$. This minimum is strictly 
	larger than~$-1$ because $F_K(z)>0$ for all $z\in\R$.

	To find the minimum, we compute the derivative of $F_K(z)-\Phi(z/a)$ with 
	respect to~$z$. It is not difficult to verify that the minimum is attained 
	for
	\[
		z = z_{\min} = K - \frac{b}{c} \sqrt{K^2 + c^2 \log(c^2/b^2)},
	\]
	from which it follows that $\inf_{z\in\R} F_K(z) - \Phi(z/a) + 1 = P_K$, 
	with $P_K$ as defined in~\eqref{eqn:defPK}. From the remarks above we know 
	that $0<P_K<1$.
\end{proof}

\subsection[Exact conditioning]{Conditioning on exactly $k_n$ successes}

For the sake of completeness, we finally treat the case of conditioning on the 
total number of successes being equal to~$k_n$. The situation is not very 
interesting here.

\begin{theorem}
	Let $\hat\X_n$ be a random vector having the conditional distribution 
	of~$\X_n$, conditioned on the event $\{\Sigma_n=k_n\}$. Define $\hat\Y_n$ 
	similarly. If all $\beta_i$ ($i\in \{1,\dots,M\}$) are equal, then $\hat 
	\X_n$ and~$\hat \Y_n$ have the same distribution for every $n\geq 1$. 
	Otherwise, $\sup \Pr(\hat \X_n = \hat \Y_n)\to 0$ as $n\to\infty$.
\end{theorem}

\begin{proof}
	If all $\beta_i$ ($i\in\{1,\dots,M\}$) are equal, then by 
	Proposition~\ref{prop:equallaws} we have that $\hat \X_n$ and~$\hat \Y_n$ 
	have the same distribution for every $n\geq 1$. If $\beta_i\neq \beta_j$ 
	for some $i,j\in\{1,\dots,M\}$, then it can be shown that $\sup\Pr(\hat 
	\X_n \leq \hat \Y_n)\to 0$ as $n\to\infty$, by a similar argument as in 
	the proof of Lemma~\ref{lem:abovehatalpha}; instead of 
	Theorem~\ref{thm:lln} use Lemma~\ref{lem:centres}.
\end{proof}

\bibliographystyle{amsplain}
\bibliography{StochasticDomination}

\end{document}